\documentclass[oneside]{amsart}
\usepackage{graphicx} 
\usepackage{comment}
\usepackage{amsmath}
\usepackage{relsize}
\usepackage{amsthm}
\usepackage{color,pxfonts,fix-cm}
\usepackage{latexsym}
\usepackage[mathletters]{ucs}
\DeclareUnicodeCharacter{32}{$\ $}
\usepackage[T1]{fontenc}
\usepackage[utf8x]{inputenc}
\usepackage{pict2e}
\usepackage{wasysym}
\usepackage[english]{babel}
\usepackage{tikz}
\usepackage{tikz-cd}
\usepackage[a4paper]{geometry}
\usepackage{parskip}
\usepackage[numbers,sort&compress]{natbib}
\usepackage{enumitem}

\theoremstyle{plain}
\newtheorem{theorem}{Theorem}[section]
\newtheorem{lemma}[theorem]{Lemma}
\newtheorem{proposition}[theorem]{Proposition}
\newtheorem{corollary}[theorem]{Corollary}
\theoremstyle{definition}
\newtheorem{definition}[theorem]{Definition}
\theoremstyle{remark}
\newtheorem{remark}[theorem]{Remark}

\makeatletter
\renewcommand\subsubsection{\@startsection{subsubsection}{3}%
  \z@{.5\linespacing\@plus.7\linespacing}{.3\linespacing}%
  {\normalfont\bfseries}}
\makeatother

\title{Rigidity of self-maps of $V_{n,2}$ and manifolds
tangentially homotopy equivalent to $V_{n,2} \times S^k$}
\author{Sagnik Biswas}
\date{}

\begin{document}

\begin{abstract}
We study two problems concerning the Stiefel manifolds $V_{n,2}$ and their products with spheres.
First, we address a rigidity problem: we determine, for most values of~$n$, all self-maps of $V_{n,2}$ that are homotopic to an almost diffeomorphism.
Second, we classify smooth closed manifolds tangentially homotopy equivalent to $V_{n,2} \times S^k$ up to almost diffeomorphism, for $k = 3, 5$ or $7 \leq k, k \neq 2^i - 2 \ \text{and} \ Dim(V_{n,2} \times S^k) \neq 2^i - 2$.
Our method is to find explicit inverses in the structure set via normal invariants of specific tangential homotopy equivalences. In favourable cases --- notably $V_{12,2} \times S^3$, $V_{16,2} \times S^3$, $V_{10,2} \times S^5$ --- the classification is complete: every such manifold is almost diffeomorphic to $V_{n,2} \mathbin{\#} \Sigma \times S^k$ for some exotic sphere $\Sigma$.
In the general case, we identify inverses for a large subgroup of $\operatorname{Im}(\eta)$ and provide a reasonable direction for the remainder.
\end{abstract}

\maketitle

\section{Introduction}
Classifying smooth manifolds up to diffeomorphism is the fundamental question in Topology. It was proved that this question is unsolvable in its full form. Some restriction to the problem is therefore required. One approach is to fix the homotopy type of the manifold and then classify all manifolds up to diffeomorphism within the fixed homotopy type. This problem is highly non-trivial in itself. However, there is a framework with which one can attack it, namely the approach using the surgery exact sequence. A natural and more tractable version of the problem is to classify all smooth manifolds up to almost diffeomorphism in a fixed tangential homotopy type. While still a difficult problem, it is considerably more approachable than the original one.

A necessary condition for two manifolds of the same homotopy type to be almost diffeomorphic is that they are tangentially homotopy equivalent. In general, this condition is not sufficient to guarantee almost diffeomorphism, as there are examples of manifolds that are tangentially homotopy equivalent but not homeomorphic~\cite{ottenberger,crowley:7mflds}. However, in the following cases~\cite{desapio} tangential homotopy equivalence is shown to be sufficient:

\begin{enumerate}
\item We are classifying in the homotopy type of an $(n-1)$-connected $2n$-manifold.
\item We are classifying in the homotopy type of an $(n-1)$-connected $(2n+1)$-manifold which is stably parallelizable, $H_n$ is cyclic, and $n$ is even.
\end{enumerate}

In this paper, we fix our manifold to be $M_{n,k} = V_{n,2} \times S^k$, where $n$ and $k$ vary. Our objective is to investigate how manifolds in the tangential homotopy type of $M_{n,k}$ look. In doing so, we show that in many cases all manifolds tangentially homotopy equivalent to $M_{n,k}$ are homeomorphic to $M_{n,k}$.

We also investigate the rigidity problem for the manifolds $V_{n,2}$: do there exist self-maps on these manifolds that can be homotoped to an almost diffeomorphism? What are necessary and sufficient conditions for a self-map to be homotopic to an almost diffeomorphism? Such maps must be self-homotopy equivalences, so this is equivalent to asking the same question for self-homotopy equivalences rather than arbitrary self-maps. For an infinite collection of $V_{n,2}$, we completely classify all such self-homotopy equivalences.

The main results of this paper are the following.

\noindent\textbf{Theorem 4.2} (Rigidity for $V_{2n+1,2}$, \S4).
Let $P_g$ be a pinch map of $V_{2n+1,2}$ where $g\colon S^{4n-1}\to S^{2n-1}\cup_2 e^{2n}$.
\begin{enumerate}[label=\arabic*)]
\item When $2n\equiv 4,6\pmod{8}$: $P_g$ is homotopic to an almost diffeomorphism if and only if
  $\{g\}=i^*\{\alpha\}$ for some $\alpha\in\pi^{s}_{4n-1}(S^{2n-1})$ with $\{\alpha\}$ divisible by~$2$.
\item When $2n\equiv 0\pmod{8}$: $P_g$ is homotopic to an almost diffeomorphism if and only if
  $\{g\}=i^*\{\alpha\}$ and either $HD\{\alpha\}$ or $HD\{\alpha\}-J(\xi)$ is divisible by~$2$.
\end{enumerate}

\noindent\textbf{Theorem 4.4} (Rigidity for $V_{2n+2,2}$, \S4).
Let $g=\alpha\vee\beta\colon S^{4n+1}\to S^{2n}\vee S^{2n+1}$ and $P_g$ the corresponding pinch map of $V_{2n+2,2}$.
Then $P_g$ is homotopic to an almost diffeomorphism if and only if $\{\alpha\},\{\beta\}\in\operatorname{Im}\,J$.

\noindent\textbf{Corollary 5.3} (Classification for $V_{2n+1,2}\times S^k$, \S5.A).
Let $V=V_{2n+1,2}$ and $k=3,5$ or $7\le k$, $k\ne 2^i-2$ and, $Dim(V_{n,2} \times S^k) \neq 2^i - 2$.
If $\Omega_1$ and $\Omega_2$ is surjective onto $\operatorname{Im}(\tau_*)$ ,
then every smooth closed manifold tangentially homotopy equivalent to $V\times S^k$
is almost diffeomorphic to $(V\mathbin{\#}\Sigma)\times\Sigma^k$ for some $\Sigma\in\Theta_{4n-1}$, $\Sigma^k\in\Theta_k$.

\noindent\textbf{Corollary 5.6} (Classification for $V_{2n+2,2}\times S^k$, \S5.B).
Let $V=V_{2n+2,2}$ and $k=3,5$ or $7\le k$, $k\ne 2^i-2$ and $Dim(V_{n,2} \times S^k) \neq 2^i - 2$.
If $\Psi_1$ and $\Psi_2$ are surjective onto $Im(\tau_*)$, then every smooth closed manifold
tangentially homotopy equivalent to $V\times S^k$ is almost diffeomorphic to $(V\mathbin{\#}\Sigma)\times\Sigma^k$.
In particular this holds for $V_{12,2}\times S^3$, $V_{16,2}\times S^3$ and $V_{10,2}\times S^5$.

\noindent\textbf{Organisation of the paper.}
Section~2 establishes the necessary background: the surgery exact sequence and normal invariants, the normal invariant criterion for almost diffeomorphisms, Spanier--Whitehead duality, and Crowley's framework for computing normal invariants of pinch maps.
In Section~3 we use this framework to explicitly compute the normal invariants of self-homotopy equivalences of $V_{n,2}$, treating the families $V_{2n+2,2}$ (\S3.1) and $V_{2n+1,2}$ (\S3.2) in turn.
Section~4 applies these computations to a rigidity problem: we determine, in most cases completely, all self-maps of $V_{n,2}$ that are homotopic to an almost diffeomorphism --- not only the pinch maps. Theorem~4.2 settles this for $V_{2n+1,2}$ and Theorem~4.4 for $V_{2n+2,2}$.
Section~5 addresses the classification of smooth closed manifolds tangentially homotopy equivalent to $V_{n,2} \times S^k$ up to almost diffeomorphism. Our approach is to find explicit inverses of elements in $\operatorname{Im}(\eta) \subset [V \times S^k, G/O]$ via normal invariants of specific tangential homotopy equivalences. In favourable cases we find inverses for all of $\operatorname{Im}(\eta)$, yielding a complete classification (Corollaries~5.3 and~5.6). In the general case we find inverses for a large subgroup of $\operatorname{Im}(\eta)$ and provide a possible way forward for approaching the remainder (Theorem~5.1 and Remark~5.4). Section~5.1 treats $V = V_{2n+1,2}$ and Section~5.2 treats $V = V_{2n+2,2}$.

\section{Preliminaries}

\subsection{Surgery exact sequence and normal invariants}

Throughout, $M^m$ denotes a smooth, closed, oriented, simply connected manifold of dimension $m \geq 5$.

\subsubsection*{Structure set and normal invariants}
We use the standard notation: $G$, $SG$ for the stable monoid of self-homotopy equivalences of spheres and its degree-$1$ component; $O$, $SO$ for the stable orthogonal groups; and $G/O$ for the homotopy fibre of $O \to G$, which classifies stable vector bundles with a fibre homotopy trivialisation of the associated sphere bundle. We refer to~\cite{ranicki} for details.

The \emph{structure set} $hS(M)$ is the set of $h$-cobordism classes of homotopy equivalences $f : N \xrightarrow{\simeq} M$~\cite{browder,wall}. The \emph{normal invariant map} $\eta : hS(M) \to [M, G/O]$ assigns to $f$ the classifying map of the stable bundle $(f^{-1})^*\nu_N - \nu_M$ together with its canonical fibre homotopy trivialisation~\cite[\S 3]{ranicki}. The set of normal invariants $\mathcal{N}(M) \cong [M, G/O]$~\cite[Proposition 9.43]{ranicki}.

\subsubsection*{Surgery exact sequence}
For a simply connected manifold $M^m$, $m \geq 5$, the \emph{surgery exact sequence}~\cite{browder,wall} is:
\[ \cdots \to
L_{m+1}(\mathbb{Z}) \xrightarrow{\;\partial\;} hS(M) \xrightarrow{\;\eta\;} [M,\, G/O] \xrightarrow{\;\sigma\;} L_m(\mathbb{Z})
\]
Here $\sigma$ is the surgery obstruction and the Wall groups $L_n(\mathbb{Z})$ are the standard $4$-periodic groups.

\subsubsection*{Tangential structure set and tangential surgery exact sequence}
A \emph{tangential homotopy equivalence} $f : N \xrightarrow{\simeq} M$ is a homotopy equivalence such that $f^* \nu_M \cong \nu_N$. The \emph{tangential structure set} $hS^t(M) \stackrel{\mathrm{def}}{=} \{ (f,b) \mid f:N \to M,\, b: \nu_N \to \nu_M \}/{\simeq}$. Here $f$ is a tangential homotopy equivalence and $b$ is a bundle map covering $f$ that is a fibrewise isomorphism.

To define the tangential normal invariant map $\eta^t$, we need the infinite loop space $QS^0$ and its additive structure. We follow~\cite[\S 6]{crowley:kervaire}. Define $QS^0 := \operatorname{colim}_{N} \, \Omega^N S^N$, with path components $(QS^0)_k$ indexed by degree. The loop-space multiplication $* : QS^0 \times QS^0 \to QS^0$ adds degrees on $\pi_0$ and makes $[X, QS^0]$ into an abelian group for any finite CW complex $X$. All path components are homotopy equivalent via translation; in particular, the $[1]*$ action provides a bijection:
\[
 [1]* \;:\; [X,\, (QS^0)_0] \;\xrightarrow{\;\sim\;}\; [X,\, (QS^0)_1] = [X,SG].
\]
Given a tangential homotopy equivalence $(f,b) \in hS^t(M)$, the Pontryagin--Thom construction yields a Thom-space level map $\operatorname{Th}(b) : \operatorname{Th}(\nu_N) \to \operatorname{Th}(\nu_M)$. Composing with the collapse map $S^{m+K} \to \operatorname{Th}(\nu_N)$ and taking the Spanier--Whitehead dual gives a stable cohomotopy element $\tilde{\eta}^t(f,b) \in \{M_+, S^0\}_*$. Via the loop-suspension adjunction and the pointed-to-free isomorphism, this lands in the degree-one component of $QS^0$, defining the \emph{tangential normal invariant}~\cite{madsen}:
\[
\eta^t(f, b) \;\in\; [M,\, (QS^0)_1] \;\cong\; [M,\, SG].
\]

These fit into the \emph{tangential surgery exact sequence} :
\[
L_{m+1}(\mathbb{Z}) \xrightarrow{\;\partial\;} hS^t(M) \xrightarrow{\;\eta^t\;} [M,\, SG] \xrightarrow{\;\sigma\;} L_m(\mathbb{Z}).
\]

\subsubsection*{Relating the two sequences}
There is a fibration $SO \to SG \xrightarrow{\tau} G/O$, inducing a map $\tau_* : [M, SG] \to [M, G/O]$. This fits into a map of exact sequences:
\[
\begin{tikzcd}[row sep=2em, column sep=2.5em]
L_{m+1}(\mathbb{Z}) \arrow[r,"\partial"] \arrow[d,"="]
& hS^t(M) \arrow[r,"\eta^t"] \arrow[d,"F"]
& {[M, SG]} \arrow[r,"\sigma"] \arrow[d,"\tau_*"]
& L_m(\mathbb{Z}) \arrow[d,"="] \\
L_{m+1}(\mathbb{Z}) \arrow[r,"\partial"]
& hS(M) \arrow[r,"\eta"]
& {[M, G/O]} \arrow[r,"\sigma"]
& L_m(\mathbb{Z})
\end{tikzcd}
\]
In particular, $\tau_*(\eta^t(f,b)) = \eta(f)$, and $\operatorname{Im}(\tau_*) \subset [M, G/O]$ consists precisely of the normal invariants of tangential homotopy equivalences. Here $F$ is the forgetful map.

\subsubsection*{Composition formula for normal invariants}
Let $N \xrightarrow{f} P \xrightarrow{g} M$ be compositions of homotopy equivalences. Madsen's composition formula~\cite[\S 2.5--2.6]{madsen} gives:
\[
\eta(g \circ f) \;=\; \eta(g) \;+\; (g^{-1})^{*}\,\eta(f) \;\;\in\; [M,\; G/O].
\]
The same formula holds for tangential normal invariants $\eta^t$ with $SG$ in place of $G/O$.

\subsection{Normal invariant criterion for almost diffeomorphisms}

\subsubsection*{Almost diffeomorphisms}
A homeomorphism $f : M_1 \to M_2$ between closed smooth $m$-manifolds is called an \emph{almost diffeomorphism} if there exist points $p \in M_1$, $q \in M_2$ such that the restriction
\[
f|_{M_1 \setminus \{p\}} \;:\; M_1 \setminus \{p\} \;\xrightarrow{\;\cong\;}\; M_2 \setminus \{q\}
\]
is a diffeomorphism. Two manifolds $M_1$ and $M_2$ are \emph{almost diffeomorphic} if such a map exists.

\subsubsection*{The normal invariant criterion for almost diffeomorphism}

\begin{proposition}
\label{prop:cstar}
Let $M^m$ be a smooth, closed, oriented, simply connected manifold with $m \geq 5$, and let $\bar{f} : M \xrightarrow{\simeq} M$ be a homotopy equivalence. Then $\bar{f}$ is homotopic to an almost diffeomorphism if and only if
\[
\eta(\bar{f}) \;\in\; \operatorname{im}\bigl(c^* : [S^m,\, G/O] \to [M,\, G/O]\bigr).
\]
\end{proposition}

\begin{proof}
$(\Rightarrow)$\; Suppose $\bar{f}$ is homotopic to an almost diffeomorphism $f$. By~\cite[Theorem~3.1]{crowley:7mflds}, there is a diffeomorphism $f' : M \mathbin{\#} \Sigma \xrightarrow{\cong} M$ such that $f' = f$ on $M \setminus D$. The map $f \# h_\Sigma \colon M \mathbin{\#} \Sigma \to M$ coincides with $f' $ on $M \setminus D$. Since both are homeomorphisms agreeing outside a disc, they are topologically isotopic~\cite{alexander}. Hence in $hS(M)$ we have $[M \# \Sigma, f \# h_\Sigma] = [M \# \Sigma, f']$. This gives $0 = \eta(f') = \eta(f \# h_\Sigma) = \eta(f) + c^*(\eta([\Sigma]))$, so $\eta(f) \in \operatorname{Im}\, c^*$.

\medskip
$(\Leftarrow)$\; Let $\eta(\bar{f}) \in \operatorname{Im}\, c^*$. Then $\eta(\bar{f}) = c^*\eta([\Sigma]) = \eta(h_\Sigma)$ (using surjectivity of the normal invariant map for $S^m$). Hence $[\bar{f}] = [h_{\Sigma \# \Sigma'}] \in hS(M)$, which yields a homotopy-commutative diagram:
\[
\begin{tikzcd}[column sep=3em]
M \arrow[r, "f"] \arrow[d, "\phi^{-1}"', "\cong"] & M \\
M \mathbin{\#} \Sigma \mathbin{\#} \Sigma' \arrow[ur, "h_{\Sigma \mathbin{\#} \Sigma'}"'] &
\end{tikzcd}
\]
The lower composed map is an almost diffeomorphism, so $\bar{f}$ is homotopic to an almost diffeomorphism.
\end{proof}

\subsection{Spanier--Whitehead duality}

We recall the basic setup of Spanier--Whitehead duality; see~\cite[I.4]{browder} and~\cite{funcspaces} for details. Two finite pointed CW complexes $X$ and $Y$ are \emph{$n$-dual} if there exists a duality map $\mu : S^n \to X \wedge Y$ inducing slant-product isomorphisms $\widetilde{H}^r(X) \xrightarrow{\cong} \widetilde{H}_{n-r}(Y)$ for all $r$. The \emph{$n$-dual} of $X$, denoted $D_n X$, is unique up to stable equivalence. Duality extends to maps and reverses their direction: for finite pointed CW complexes $X$ and $Y$ there is a natural isomorphism
\[
[\Sigma^k X,\, \Sigma^k Y] \;\cong\; [\Sigma^l D_n Y,\, \Sigma^l D_n X]
\]
in the stable range. We use several standard properties of the duality functor~\cite[Thm~6.1, 6.5, Cor~6.7]{funcspaces} (duals of wedges, suspensions, cofibre sequences, self-duality of stable maps between spheres, and the homology--cohomology interchange); these will be recalled as needed in \S\,3.

\subsection{Normal invariants of pinch maps}

\providecommand{\Dhat}{\widehat{\vphantom{D^{\ast\ast}}D}}

We summarise the framework of \textbf{\cite[\S 6--7]{crowley:kervaire}} for computing the tangential normal invariants of pinch maps. Throughout, $X^m$ is a smooth, closed, oriented, simply connected manifold with $m \geq 5$, embedded in $\mathbb{R}^{m+K}$ for $K$ large.

\subsubsection*{Pinch maps}
Let $X^m$ have CW-structure with top cell $e^m$ attached by $\rho : S^{m-1} \to X^{(m-1)}$. Given a map $x : S^m \to X$ factoring through the lower skeleton (i.e.\ $x$ factors as $S^m \xrightarrow{y} Y \xrightarrow{t} X$ for some sub-skeleton or submanifold $Y \subset X$), the \emph{pinch map} $p(x) : X \to X$ is defined as:
\[
p(x) \;:\; X \;\xrightarrow{\;q\;}\; X \vee S^m \;\xrightarrow{\;\mathrm{Id} \vee x\;}\; X \vee X \;\xrightarrow{\;\nabla\;}\; X,
\]
where $q$ is the quotient map collapsing the boundary of the top cell to a point, and $\nabla$ is the fold map.  Since $x$ factors through the lower skeleton, $p(x)$ induces the identity on all homology groups and hence is a degree-one homotopy equivalence. Moreover, the tangent bundle is preserved: if $x$ pulls back the stable normal bundle $\nu_X$ trivially (which happens whenever $X$ is stably parallelizable), then $p(x)$ is a \emph{tangential homotopy equivalence}~\cite[\S 7]{crowley:kervaire}.

\subsubsection*{Normal bordism and the tangential normal invariant of a sum}
A normal map over $X$ defines a bordism element $[P, h, b] \in \Omega_m(X, \nu_X)$, which the Pontryagin--Thom construction identifies with a class $\mu_X([P,h,b]) = \rho_b \in \pi_{m+K}(\operatorname{Th}(\nu_X))$. The sum $[X, \mathrm{Id}, \mathrm{Id}] + [P, h, b]$ corresponds geometrically to a pinch map when $[P, h, b]$ arises from a map $x : S^m \to X$~\cite[Proof of 7.4]{crowley:kervaire}. The Thom space $\operatorname{Th}(\nu_X)$ is $S$-dual to $X_+$~\cite[I.4]{browder}, giving an isomorphism
\[
\Dhat \;:\; \pi_{m+K}(\operatorname{Th}(\nu_X)) \;\xrightarrow{\;\cong\;}\; \{X_+,\, S^0\}_* \;\cong\; [X_+,\, (QS^0)]_* \;\cong\; [X,\, QS^0].
\]

\begin{lemma}[{\cite[Lemma~6.5]{crowley:kervaire}}]
\label{lem:eta-connected-sum}
Let $[P, h, b] \in \Omega_m(X, \nu_X)_0$ be a degree-0 normal bordism element with Pontryagin--Thom class $\rho_b = \mu_X([P, h, b]) \in \pi_{m+K}(\operatorname{Th}(\nu_X))_0$. Then the tangential normal invariant of the ``connected sum with identity'' is:
\[
\eta^t\bigl([X,\, \mathrm{Id},\, \mathrm{Id}] + [P,\, h,\, b]\bigr) \;=\; [1] * \Dhat(\rho_b) \;\in\; [X,\, SG].
\]
This translates the tangential normal invariant of a normal degree-$0$ map inside $[X,(\Omega^\infty S^\infty)_0]$ to the normal invariant of the corresponding pinch map inside $[X,SG]$.
\end{lemma}

\subsubsection*{Setup for the computation of normal invariant of pinch maps}
We now specialise to the situation where the degree-0 normal bordism element comes from a map through a submanifold. Let $t : Y^{m-l} \hookrightarrow X^m$ be a codimension-$l$ submanifold ($l > 0$) with normal bundle $\nu_t$, and let $t^! : X_+ \to \operatorname{Th}(\nu_t)$ be the Pontryagin--Thom collapse (umkehr) map for the embedding $t$ and $+$ going to basepoint.

We suppose that we are given a map $y : S^m \to Y$ such that the composite $x = t \circ y$,
\[
x \;:\; S^m \;\xrightarrow{\;y\;}\; Y \;\xrightarrow{\;t\;}\; X,
\]
pulls back $\nu_X$ trivially. Since $\nu_{S^m}$ is trivial, this is equivalent to assuming the existence of a bundle map $b_y : \nu_{S^m} \to t^*(\nu_X)$. If $b_t : t^*(\nu_X) \to \nu_X$ is the canonical bundle map, we set $b_x := b_t \circ b_y$ and consider the following diagram of bundle maps:
\[
\begin{tikzcd}[row sep=2.5em, column sep=3em]
\nu_{S^m}
  \arrow[r, "b_y"]
  \arrow[d]
& t^*(\nu_X)
  \arrow[r, "b_t"]
  \arrow[d]
& \nu_X
  \arrow[d]
\\
S^m
  \arrow[r, "y"]
& Y
  \arrow[r, "t"]
& X
\end{tikzcd}
\]
The homotopy class $\rho_x := \mu_X([S^m, x, b_x])$ is then given as the composite of Thom-space maps:
\begin{equation}\label{eq:rhox}\tag{2.1}
\rho_x \;=\;\;:\; S^{m+K} \;\xrightarrow{\;\rho_y\;}\; \operatorname{Th}\bigl(t^*(\nu_X)\bigr) \;\xrightarrow{\;\operatorname{Th}(b_t)\;}\; \operatorname{Th}(\nu_X),
\end{equation}
where $\rho_y$ is the homotopy class $\operatorname{Th}(b_y)_*(\rho_{S^m}) \in \pi_{m+K}(\operatorname{Th}(t^*(\nu_X)))$, and $\operatorname{Th}(b_t)$, $\operatorname{Th}(b_y)$ denote the induced maps of Thom spaces. Since $\rho_x$ has degree zero, we have the map $\Dhat(\rho_x) : X_+ \to (QS^0)_0$ is defined. To analyse $\Dhat(\rho_x)$ we consider the $S$-duals of the maps in~\eqref{eq:rhox}.

\begin{lemma}[{\cite[Lemma~7.3]{crowley:kervaire}}]
\label{lem:SW-dual-factoring}
In the above setup, with $\rho_x = \operatorname{Th}(b_t) \circ \rho_y$ as in~\eqref{eq:rhox}:
\begin{enumerate}[label=(\roman*)]
\item The $S$-dual of the Thom-level bundle map $\operatorname{Th}(b_t) : \operatorname{Th}(t^*\nu_X) \to \operatorname{Th}(\nu_X)$ is the collapse (umkehr) map $t^! : X_+ \to \operatorname{Th}(\nu_t)$.
\item The duality isomorphism identifies $\pi_{m+K}(\operatorname{Th}(t^*\nu_X))$ with $[\operatorname{Th}(\nu_t),\, (QS^0)_0]$.
\item The dual element factors as:
\[
\Dhat(\rho_x) \;=\; \Dhat(\rho_y) \circ t^! \;\in\; [X_+,\, (QS^0)_0].
\]
\end{enumerate}
\end{lemma}

\begin{lemma}[{\cite[Lemma~7.4]{crowley:kervaire}}]
\label{lem:eta-pinch}
There exists a bundle map $b : \nu_X \to \nu_X$ covering the pinch map $p(x) : X \to X$ such that the tangential normal invariant is:
\[
\eta^t\bigl(p(x),\, b\bigr) \;=\; [1] * \Dhat(\rho_x) \;=\; [1] * \bigl(\Dhat(\rho_y) \circ t^!\bigr) \;\in\; [X,\, SG].
\]
\end{lemma}

\subsubsection{About normal invariant of certain kinds of pinch map on $V \times S^k$}
Let $\alpha := i \circ a \colon S^{m+k} \xrightarrow{a} V \xrightarrow{i} V \times S^k$, where $i$ is inclusion into a fibre. The \textbf{pinch map} $p_\alpha \colon V \times S^k \to V \times S^k$ is the composition:
\[
V \times S^k \;\xrightarrow{\;\operatorname{pinch}\;}\; V \times S^k \vee S^{m+k} \;\xrightarrow{\;\mathrm{Id} \,\vee\, \alpha\;}\; V \times S^k \vee V \times S^k \;\xrightarrow{\;\nabla\;}\; V \times S^k,
\]
Here we collapse the boundary of an embedded disc $D^{m+k} \hookrightarrow \operatorname{int}(V \times D^k_-)$ to a point, creating a wedge summand $S^{m+k}$.

Decompose $S^k = D^k_- \cup_{S^{k-1}} D^k_+$, giving $V \times S^k = {V \times D^k_-} \cup {V \times D^k_+}$.
Since the pinch is performed from the interior of $V \times D^k_- \hookrightarrow V \times S^k$, the entire construction (pinch, wedge, fold) takes place in the interior of $M_1 = V \times D^k_-$. On $M_2 = V \times D^k_+$, none of these operations have any effect.
In other words, we can do the \emph{same} pinch-wedge-fold construction on $V \times D^k_-$ alone:
\[
P_\alpha\big|_{M_1} \;:\; V \times D^k_- \;\xrightarrow{\;\operatorname{pinch}\;}\; V \times D^k_- \vee S^{m+k} \;\xrightarrow{\;\mathrm{Id} \,\vee\, \alpha\;}\; V \times D^k_- \vee V \times S^k \;\xrightarrow{\;\nabla'\;}\; V \times D^k_- \subset V \times S^k,
\]
which is a self-homotopy equivalence of $(V \times D^k_-,\; V \times S^{k-1})$ relative to the boundary.
The pinch map on all of $V \times S^k$ is then:
\[
P_\alpha \;=\; P_\alpha\big|_{M_1} \;\cup\; \mathrm{Id}_{M_2}.
\]
That is, $P_\alpha$ on $V \times S^k$ is exactly $P_\alpha|_{V \times D^k_-}$ capped on top by $V \times D^k_+ \xrightarrow{\mathrm{Id}} V \times D^k_+$. By~\cite[Prop.~2.3]{bks}, $\eta(P_\alpha)$ lies in the $[\Sigma^k V \vee S^k, G/O]$ component.

The tangential normal map covering $P_\alpha$ is the pair $(P_\alpha,\, \overline{P_\alpha})$, where the covering bundle map (following~\cite[proof of Lemma~7.4]{crowley:kervaire}) is:
\[
\overline{P_\alpha} \;=\; \mathrm{Id}(\mathcal{\mathlarger{\mathlarger {\nu}}}_{(V \times S^k)}) \;\mathbin{\#}\; (b_i \circ b_a).
\]
Here $b_a$ is the Pontryagin--Thom bundle map on the attaching map $a$ and $b_i$ is the bundle map on the inclusion $i \colon V \hookrightarrow V \times S^k$.
Crucially, the connected-sum modification $b_i \circ b_a$ is supported entirely inside $M_1 = V \times D^k_-$. On $M_2 = V \times D^k_+$:
\[
\overline{P_\alpha}\big|_{M_2} \;=\; \mathrm{Id}(\mathlarger{\mathlarger {\nu}}_{V \times S^k})\big|_{M_2} \;=\; \mathrm{Id}_{\mathlarger{\nu}_{M_2}}.
\]
So the tangential normal map $(P_\alpha, \overline{P_\alpha})$ is again of the form:
\begin{center}
\emph{non-trivial tangential map on $V \times D^k_-$, capped by identity (map and bundle map) on $V \times D^k_+$.}
\end{center}
The same Proposition~\ref{prop:cstar} (applied in the tangential surgery exact sequence with $SG$ in place of $G/O$) gives:
\[
\eta^t(P_\alpha,\, \overline{P_\alpha}) \;=\; c^*\,\eta^t\bigl((P_\alpha, \overline{P_\alpha})\big|_{M_1}\bigr) \;\in\; \operatorname{im}\bigl(c^* \colon [\Sigma^k V \vee S^k,\; SG] \to [V \times S^k,\; SG]\bigr).
\]
\begin{remark}
\label{rem:identity-on-M2}
The key point in both cases is the same: both the map $P_\alpha$ \emph{and} the covering bundle map $\overline{P_\alpha}$ are the identity on $M_2 = V \times D^k_+$. For the map, this is by construction (the pinch takes place only in the lower hemisphere). For the bundle map, this holds because $\overline{P_\alpha} = \mathrm{Id}(\mathlarger{\mathlarger {\nu}}) \mathbin{\#} (b_i \circ b_a)$, and the connected-sum modification is localised inside $M_1$.
\end{remark}
\section{Normal invariants of self homotopy equivalences of $V_{n,2}$}

Any homotopy equivalence of $V_{n,2}$ is a post-composition of a diffeomorphism with a pinch map~\cite[Theorem~2.4]{nomura}, so computing normal invariants of pinch maps suffices. In \S\,3.1 we treat the family $V_{2n+2,2}$, where the splitting $\pi_{4n+1}(V) \cong \pi_{4n+1}(S^{2n}) \oplus \pi_{4n+1}(S^{2n+1})$ allows a direct computation. In \S\,3.2 we treat $V_{2n+1,2}$, where the Moore-space skeleton $S^{2n-1} \cup_2 e^{2n}$ requires a more involved analysis.

\subsection{Calculating Normal Invariants for $V_{2n+2,2}$}

We have the fibration
\[
S^n \xrightarrow{i} V_{2n+2,2} \underset{s}{\overset{\pi}{\rightleftarrows}} S^{n+1}
\]
We also have a CW structure $V_{2n+2,2} \simeq (S^{2n} \vee S^{2n+1}) \cup_{\rho} e^{4n+1}$, where $\rho$ is stably trivial~\cite[Theorem~7.10]{james}.

There is a short exact sequence
\[
1 \to \pi_{2b-3}(S^{b-1}) + \pi_{2b-3}(S^{b-2}) / H \to \mathcal{E}(V_{b,2}) \to \mathbb{Z}_2 \times \mathbb{Z}_2 \to 1,
\]
where $H$ is the subgroup generated by $J(\xi \eta_{b-2})$ and the Whitehead product $[\eta_{b-2}^2, \iota_{b-2}]$, and $b = 2n+2$~\cite[Theorem~2.4]{nomura}. Here the first map is given by pinch maps and the last component corresponds to diffeomorphisms. We compute the normal invariants of the pinch maps first.

\subsubsection*{3.1.A.\ Normal invariants of pinch maps using fibration}

Note that $\pi_{4n+1}(S^{2n}) \oplus \pi_{4n+1}(S^{2n+1}) \xrightarrow[\cong]{l^*+s^*} \pi_{4n+1}(V)$ is an isomorphism, since $V$ admits a section of the sphere bundle. Hence any map $S^{4n+1} \to V$ factors through lower-dimensional spheres and has degree $0$, so every element of $\pi_{4n+1}(V)$ gives a tangential homotopy equivalence as a pinch map. We follow the setup of~\cite{crowley:kervaire}. 

\begin{lemma}\label{lem:3.1}
Let $S^{4n+1} \xrightarrow{\alpha} S^{2n}$ be an element of $\pi_{4n+1}(S^{2n})$ and let $p_{(l \circ \alpha)}$ denote the corresponding pinch map.
\begin{enumerate}[label=\textup{(\roman*)}]
\item $i^*\bigl(\tilde{\eta}^t(p_{(l \circ \alpha)},\mathrm{Id} \# b_l \circ b_\alpha)\bigr) = (0,\pm \alpha) \in \{ S^{2n},S^0\} \oplus \{ S^{2n+1},S^0\}$.
\item $i^*\bigl(\eta^t(p_{(l \circ \alpha)},\mathrm{Id} \# b_l \circ b_\alpha)\bigr) = (0,\pm \alpha) \in [S^{2n},SG] \oplus [S^{2n+1},SG]$.
\end{enumerate}
Here we are denoting the image of $x \in \{V,S^0\}$ again by $x \in [V,SG]$.
\end{lemma}

\begin{proof}[proof of (i)]
Let $S^{4n+1} \xrightarrow{\alpha} S^{2n}$ be an element of $\pi_{4n+1}(S^{2n})$ and let $p_{(l \circ \alpha)}$ denote the corresponding pinch map, which is a tangential homotopy equivalence of degree one. Embed $V \subset \mathbb{R}^{4n+1+k}$ for $k \gg 4n+1$. We have the following diagram:
\begin{equation}\label{eq:3.1}\tag{3.1}
\begin{tikzcd}
S^{4n+1} \times \mathbb{R}^k 
  \arrow[r,"\alpha \times Id"] 
  \arrow[d, "\psi","\cong"{left}] 
& S^{2n} \times \mathbb{R}^{k} 
  \arrow[r,"l \times Id"] 
  \arrow[d,"\phi","\cong" {left}] 
& V \times \mathbb{R}^{k} 
  \arrow[d,"{\tilde{\phi}}","\cong" {left}] 
\\
\nu_{S^{4n+1}} 
  \arrow[r, "b_\alpha"'] 
  \arrow[d]
& l^*(\nu_V) 
  \arrow[r, "b_l"'] 
  \arrow[d]
& \nu_V
  \arrow[d]
\\
S^{4n+1} 
  \arrow[r,"\alpha"] 
& S^{2n} 
  \arrow[r,"l"] 
& V 
\end{tikzcd}
\end{equation}
This leads us to the following diagram:
\begin{equation}\label{eq:3.2}\tag{3.2}
\begin{tikzcd}[column sep=3em,row sep=2.5em]
S^{4n+1+k}
  \arrow[r]
  \arrow[d,dotted,"D"]
&
\operatorname{Th}\bigl(\nu_{S^{4n+1}}\bigr)
  \arrow[r, "\operatorname{Th}(b_{\alpha})"]
  \arrow[d,dotted,"D"]
&
\operatorname{Th}\bigl(l^{*}\nu_V\bigr)
  \arrow[r, "\operatorname{Th}(b_l)"]
  \arrow[d,dotted,"D"]
&
\operatorname{Th}(\nu_V)
  \arrow[d,dotted,"D"]
\\
S^{k}
&
\Sigma^{k} S^{4n+1+k}_{+}
  \arrow[l, "\Sigma^{k} c_{\theta_0}"']
&
\Sigma^{k} \operatorname{Th}(\nu_{l})
  \arrow[l, "D(\operatorname{Th}(b_{\alpha}))"']
&
\Sigma^{k} V_{+}
  \arrow[l, "\Sigma^{k} l^{!}_{+}"']
\end{tikzcd}
\end{equation}
Here $D$ denotes the Spanier--Whitehead dual $D_{4n+1+2k}$~\cite[I.4]{browder}. This is not a commutative diagram; the dotted arrows indicate Spanier--Whitehead duals. By~\cite[Lemmas~7.3,~7.4]{crowley:kervaire}, the normal invariant of $p_{(l \circ \alpha)}$ is $\eta^t(p_{(l \circ \alpha)},\mathrm{Id} \# b_l \circ b_\alpha) = [1]*g_{(l \circ \alpha)}$, where $g_{(l \circ \alpha)} \in [V,\Omega^\infty_0S^\infty]$ is the following map:
\begin{equation}\label{eq:3.3}\tag{3.3}
g_{(l \circ \alpha)}: V \xrightarrow{\;l^{!}\;}
   \operatorname{Th}(\nu_l)
   \xrightarrow{\;\operatorname{adj}\bigl(\Sigma^{k} c_{s_0} \circ D(\operatorname{Th}(b_{\alpha}))\bigr)\;}
   \Omega^{k} S^{k}.
\end{equation}
We are interested in the adjoint of $g_{(l \circ \alpha)} \in \{V,S^0\} \cong [\Sigma^k V,S^k]$, which we denote $\tilde{\eta}^t(p_{(l \circ \alpha)},\mathrm{Id} \# b_l \circ b_\alpha)$:
\begin{equation}\label{eq:3.4}\tag{3.4}
\tilde{\eta}^t(p_{(l \circ \alpha)},\mathrm{Id} \# b_l \circ b_\alpha) : \Sigma^{k} V
 \xrightarrow{\;\Sigma^{k} l^{!}\;}
 \Sigma^{k} \operatorname{Th}(\nu_{l})
 \xrightarrow{\;D(\operatorname{Th}(b_{\alpha}))\;}
 \Sigma^{k} S^{4n+1}_{+}
 \xrightarrow{\;\Sigma^{k} c_{s_0}\;}
 S^{k}.
\end{equation}
From the CW structure of $V$ we have a split exact sequence $0 \to\{S^{4n+1},S^0\} \xrightarrow{q^*} \{V,S^0\} \xrightarrow{i^*} \{ S^{2n},S^0\} \oplus \{ S^{2n+1},S^0\} \to 0$. We compute $i^*(\tilde{\eta}^t(p_{(l \circ \alpha)}))$ component by component. We have the following diagram:
\begin{equation}\label{eq:3.5}\tag{3.5}
\begin{tikzcd}[column sep=3.4em,row sep=3em]
\operatorname{Th}(\nu_{S^{4n+1+k}})
  \arrow[r, "{\operatorname{Th}(\psi)}"{above}, "\cong"{below}]
  \arrow[d, "{\operatorname{Th}(b_{\alpha})}"'{left}]
&
S^{4n+1}_{+} \wedge S^{k}
  \arrow[r, "{\simeq}"{above}, "q"{below}]
  \arrow[d, "{\Sigma^{k}\alpha_{+}}"'{right}]
&
S^{4n+1+k} \vee S^{k}
  \arrow[d, "{\Sigma^{k}\alpha \,\vee\, \mathrm{Id}}"'{right}]
\\
\operatorname{Th}(l^{*}\nu_V)
  \arrow[r, "{\operatorname{Th}(\phi)}"{above}, "\cong"{below}]
&
S^{2n}_{+} \wedge S^{k}
  \arrow[r, "{p}"{above}, "\simeq"{below}]
&
S^{2n+k} \vee S^{k}
\end{tikzcd}
\end{equation}
Call the upper composite $m$ and the lower composite $h$; both are homotopy equivalences. The $S^{2n}$ component of $i^*(\tilde{\eta}^t(p_{(l \circ \alpha)}))$ is (where $i$ is taken up to homotopy between the CW complex $V$ and the manifold $V$):
\begin{equation}\label{eq:3.6}\tag{3.6}
\Sigma^k S^{2n} \xrightarrow{i}\Sigma^{k} V
 \xrightarrow{\;\Sigma^{k} l^{!}\;}
 \Sigma^{k} \operatorname{Th}(\nu_{l})
 \xrightarrow{\;D(\operatorname{Th}(b_{\alpha}))\;}
 \Sigma^{k} S^{4n+1}_+
 \xrightarrow{\;\Sigma^k C_{s_0}\;}
 S^{k}
\end{equation}
Since $\Sigma^{k} \operatorname{Th}(\nu_{l}) \xrightarrow{D(h)} S^{2n+1+k} \vee S^{4n+1+k}$, the map~\eqref{eq:3.6} is zero. Hence the $S^{2n}$ component of $i^*\tilde{\eta}^t(p_{(l \circ \alpha)},\mathrm{Id} \# b_l \circ b_\alpha)$ is zero.\phantom\qedhere

We now compute the $S^{2n+1}$ component of $i^*\tilde{\eta}^t(p_{(l \circ \alpha)},\mathrm{Id} \# b_l \circ b_\alpha)$. By~\cite[Lemmas~7.3,~7.4]{crowley:kervaire} our map is the bottom composite of the following commutative diagram:
\begin{equation}\label{eq:3.7}\tag{3.7}
\begin{tikzcd}[column sep=3.5em,row sep=3.0em]
\Sigma^{k} V
  \arrow[r, "\Sigma^{k} l^{!}"]
 &
\Sigma^{k} \operatorname{Th}(\nu_{l})
  \arrow[r, "D(\operatorname{Th}(b_{\alpha}))"]
  \arrow[d, "\simeq"{left}, "D(h)"]
&
\Sigma^{k} S^{4n+1}_{+}
  \arrow[r,"\Sigma^k C_{s_0}"]
  \arrow[d,"D(m)"{left},"\simeq"{right}]
&
S^{k}
\\
\Sigma^{k} S^{2n+1}
  \arrow[u,"i"]
  \arrow[r, "p"']
&
S^{2n+1+k} \;\vee\; S^{4n+1+k}
  \arrow[r, "{\{\alpha\} \vee \mathrm{Id}}"']
&
S^{k} \vee S^{4n+1+k}
  \arrow[ru, "\tilde{c}", bend right=15]
\end{tikzcd}
\end{equation}
The lower middle map uses $D(\{\alpha\}) = \{\alpha\}$, since stable maps between spheres are canonically self-dual~\cite[Thm~6.1]{funcspaces} (we dualized diagram~\eqref{eq:3.5}). Using the homology--cohomology interchange under Spanier--Whitehead duality ($\widetilde{H}_i(X) \cong \widetilde{H}^{n-i-1}(D_n X)$, see~\cite[I.4]{browder}) to determine $p$ and $\tilde{c}$, we obtain:
We have $\operatorname{Th}(b_l) \cong \operatorname{Th}(\ell \times \mathrm{Id}) \simeq \Sigma^{k} l \vee \mathrm{Id}$. The map $H_{2n}(S^{2n}_{+}) \xrightarrow{\,(\ell \vee \mathrm{Id})_{*}} H_{2n}(V_{+})$ is an isomorphism, hence $H_{2n}\bigl(\operatorname{Th}(b_l)\bigr)$ is an isomorphism, so $H^{2n+1}(l^!)$ is an isomorphism by duality~\cite[I.4.13]{browder}. Therefore $H_{2n+1+k}(p)$ is an isomorphism, so $p = \pm\text{inclusion}$. From the isomorphism of $H_k(C_{s_0})$ we get $\tilde{c}|_{S^k} = \pm\mathrm{Id}$. The stable map is therefore $\pm\{\alpha\}$, giving 
$i^*\tilde{\eta}^t(p_{(l \circ \alpha)},\mathrm{Id} \# b_l \circ b_\alpha)$.
\end{proof}

\begin{proof}[proof of (ii)]
    We just pass to $SG$ here, the proof is immediate.
\end{proof}

We now perform analogous calculations for pinch maps associated to elements of $\pi_{4n+1}(S^{2n+1})$.

\begin{lemma}\label{lem:3.2}
Let $S^{4n+1} \xrightarrow{\beta} S^{2n+1}$ be an element of $\pi_{4n+1}(S^{2n+1})$ and let $p_{(s \circ \beta)}$ denote the corresponding pinch map.
\begin{enumerate}[label=\textup{(\roman*)}]
\item $i^*\bigl(\tilde{\eta}^t(p_{(s \circ \beta)},\mathrm{Id} \# b_s \circ b_\beta)\bigr) = (\pm \beta,0)$ or $(\pm \beta, \pm (\beta \circ \eta)) \in \{ S^{2n},S^0\} \oplus \{ S^{2n+1},S^0\}$.
\item $i^*\bigl(\eta^t(p_{(s \circ \beta)},\mathrm{Id} \# b_s \circ b_\beta)\bigr) = (\pm \beta,0)$ or $(\pm \beta, \pm (\beta \circ \eta)) \in [S^{2n},SG] \oplus [S^{2n+1},SG]$.
\end{enumerate}
\end{lemma}

\begin{proof}[proof of (i)]
Let $\beta \in \pi_{4n+1}(S^{2n+1})$. We have the following diagram:
\begin{equation}\label{eq:3.8}\tag{3.8}
\begin{tikzcd} 
S^{4n+1} \times \mathbb{R}^k \arrow[r,"\beta \times id"] \arrow[d, "\psi","\cong" {left}] 
& S^{2n+1} \times \mathbb{R}^k \arrow[r, " s \times \text{id} "] \arrow[d, "\phi","\cong"{left}] 
& V \times \mathbb{R}^k \arrow[d, "\tilde{\phi}", "\cong" {left}] \\
\nu_{S^{4n+1}} \arrow[r, "b_\beta"] \arrow[d] 
& s^*(\nu_V) \arrow[r, "b_s"] \arrow[d] 
& \nu_V \arrow[d] \\
S^{4n+1} \arrow[r, "\beta"] 
& S^{2n+1} \arrow[r, "s"] 
& V 
\end{tikzcd}
\end{equation}
Taking Thom-space level maps on upper squares, we get:
\begin{equation}\label{eq:3.9}\tag{3.9}
\begin{tikzcd}[column sep=large]
S^{4n+1+k} \arrow[r] 
  \arrow[d, "D", dotted] 
& \operatorname{Th}(\nu_{S^{4n+1}}) 
  \arrow[r, "\operatorname{Th}(b_\beta)"] 
  \arrow[d, "D",dotted] 
& \operatorname{Th}(s^* \nu_V) 
  \arrow[r, "\operatorname{Th}(b_s)"] 
  \arrow[d, "D",dotted] 
& \operatorname{Th}(\nu_V) 
  \arrow[d, "D",dotted] \\
S^k 
  & \Sigma^k S_+^{4n+1} 
      \arrow[l, "\Sigma^k c_{S^0}"'] 
  & \Sigma^k \operatorname{Th}(\nu(s)) 
      \arrow[l, "D(\operatorname{Th}(b_\beta))"'] 
  & \Sigma^k V_+ 
      \arrow[l, "\Sigma^k s_+^1"']
\end{tikzcd}
\end{equation}
Here $D = D_{4n+1+2k}$. By~\cite[Lemmas~7.3,~7.4]{crowley:kervaire}, the normal invariant of $p_{(s \circ \beta)}$ is $\eta^t(p_{(s \circ \beta)},\mathrm{Id} \# b_s \circ b_\beta) = [1]*g_{(s \circ \beta)}$, where $g_{(s \circ \beta)} \in [V,\Omega^\infty_0S^\infty]$ is:
\begin{equation}\label{eq:3.10}\tag{3.10}
g_{(s \circ \beta)}: V \xrightarrow{\;s^{!}\;}
   \operatorname{Th}(\nu(s))
   \xrightarrow{\;\operatorname{adj}\bigl(\Sigma^{k} c_{s_0} \circ D(\operatorname{Th}(b_{\beta}))\bigr)\;}
   \Omega^{k} S^{k}.
\end{equation}
The adjoint lies in $\{V,S^0\} \cong [\Sigma^k V,S^k]$~\cite[Lemmas~7.3,~7.4]{crowley:kervaire}. The map of interest (denoted $\tilde{\eta}^t(p_{(s \circ \beta)},\mathrm{Id} \# b_s \circ b_\beta)$) is:
\begin{equation}\label{eq:3.11}\tag{3.11}
\tilde{\eta}^t(p_{(s \circ \beta)},\mathrm{Id} \# b_s \circ b_\beta) = adj \, (g_{(s \circ \beta)}): \Sigma^{k} V
 \xrightarrow{\;\Sigma^{k} s^{!}\;}
 \Sigma^{k} \operatorname{Th}(\nu(s))
 \xrightarrow{\;D(\operatorname{Th}(b_{\beta}))\;}
 \Sigma^{k} S^{4n+1}_{+}
 \xrightarrow{\;\Sigma^{k} c_{s_0}\;}
 S^{k}.
\end{equation}

The $S^{2n}$ component of $i^*\tilde{\eta}^t(p_{(s \circ \beta)},\mathrm{Id} \# b_s \circ b_\beta)$ is the lower composite of the following commutative diagram (where $i$ is taken up to homotopy between the CW complex $V$ and the manifold $V$):
\begin{equation}\label{eq:3.12}\tag{3.12}
\begin{tikzcd}[column sep=3.5em,row sep=3.0em]
\Sigma^{k} V
  \arrow[r, "\Sigma^{k} s^{!}"]
 &
\Sigma^{k} \operatorname{Th}(\nu_(s))
  \arrow[r, "D(\operatorname{Th}(b_{\beta}))"]
  \arrow[d, "\simeq"{left}, "D(h')"]
&
\Sigma^{k} S^{4n+1}_{+}
  \arrow[r,"\Sigma^k C_{s_0}"]
  \arrow[d,"D(m')"{left},"\simeq"{right}]
&
S^{k}
\\
\Sigma^{k} S^{2n}
  \arrow[u,"i"]
  \arrow[r, "p' "]
&
S^{2n+k} \;\vee\; S^{4n+1+k}
  \arrow[r, "{\{\beta\} \vee \mathrm{Id}}"']
&
S^{k} \vee S^{4n+1+k}
  \arrow[ru, "\tilde{c}", bend right=15]
\end{tikzcd}
\end{equation}
Here we used that the dual of a wedge of spaces (resp.~maps) is the wedge of the duals~\cite[Theorem~6.5, Corollary~6.7]{funcspaces}. It remains to analyse $\Sigma^k s^!$ on $H_{2n+k}$, which we do via the following diagram:
\begin{equation}\label{eq:3.13}\tag{3.13}
\begin{tikzcd}[column sep=3.5em,row sep=3.0em]
H_{2n}(V_+)
  \arrow[r, "D","\cong"{below} ]
  \arrow[d, "(s^!_+)_*" ]
&
H^{2n+1+k}(\operatorname{Th}\nu_V)
  \arrow[r, "\cong"]
  \arrow[d]
&
H^{2n+1+k}(\Sigma^k V_+)
  \arrow[r, "\cong"]
  \arrow[d, "(\Sigma^k s_+)_*"]
&
H^{2n+1}(V_+)
  \arrow[d, "(s_+)_*"]
\\
H_{2n}(\operatorname{Th}\nu_(S))
  \arrow[r, "D","\cong"{below} ]
&
H^{2n+1+k}(\operatorname{Th}\,s^*\nu_V)
  \arrow[r, "\cong"]
&
H^{2n+1+k}(\Sigma^k S^{2n+1}_+)
  \arrow[r, "\cong"]
&
H^{2n+1}(S^{2n+1}_+)
\end{tikzcd}
\end{equation}
The composition $\mathrm{Id} : S^{2n+1} \xrightarrow{s} V \xrightarrow{\pi} S^{2n+1}$ induces $\mathrm{Id} : \mathbb{Z} \cong H^{2n+1}(S^{2n+1}) \xrightarrow{\pi^*} \mathbb{Z} \cong H^{2n+1}(V) \xrightarrow{s^*} \mathbb{Z} \cong H^{2n+1}(S^{2n+1})$. This forces $s^*$ to be an isomorphism, so $p'$ induces an isomorphism on $H_{2n+k}$, hence $p' = \pm\text{inclusion}$. The resulting stable map is $\{\beta\}$, which is the $S^{2n}$ component of $i^*\tilde{\eta}^t(p_{(s \circ \beta)},\mathrm{Id} \# b_s \circ b_\beta)$.

We now compute the $S^{2n+1}$ component. We have a similar diagram:
\begin{equation}\label{eq:3.14}\tag{3.14}
\begin{tikzcd}[column sep=3.5em,row sep=3.0em]
\Sigma^{k} V
  \arrow[r, "\Sigma^{k} s^{!}"]
&
\Sigma^{k} \operatorname{Th}(\nu_{(s)})
  \arrow[r, "D(\operatorname{Th}(b_{\beta}))"]
  \arrow[d, "\simeq"{left}, "D(h)"]
&
\Sigma^{k} S^{4n+1}_{+}
  \arrow[r, "\Sigma^{k} c_{S^{0}}"]
  \arrow[d, "D(m)"{left}, "\simeq"{right}]
&
S^{k}
\\
\Sigma^{k} S^{2n+1}
  \arrow[u, "i"]
  \arrow[r]
&
S^{2n+k} \;\vee\; S^{4n+1+k}
  \arrow[r, "{\{\beta\} \vee \mathrm{Id}}"]
&
S^{k} \vee S^{4n+1+k}
  \arrow[ru, bend right=15]
\end{tikzcd}
\end{equation}
The lower map is $0$ or $\pm(\beta \circ \eta)$, so
$i^*\tilde{\eta}^t(p_{(s \circ \beta)},\mathrm{Id} \# b_s \circ b_\beta) = (\pm \beta,0)$ or $(\pm \beta, \pm (\beta \circ \eta)) \in \{ S^{2n},S^0\} \oplus \{ S^{2n+1},S^0\}$.\phantom\qedhere
\end{proof}

\begin{proof}[proof of (ii)]
      we pass from $\{V,S^0\}$ to $[V,SG]$, where the tangential normal invariants lie, denoting the image of $x \in \{V,S^0\}$ again by $x \in \{V,SG\}$. Note that we have the map:
\begin{equation}\label{eq:3.15}\tag{3.15}
hS^t(V) \xrightarrow{\;\eta^t\;} [V,SG] \xrightarrow{\;i^*\;} [S^{2n},SG] \oplus [S^{2n+1},SG]
\end{equation}
It remains to show that all elements of $[S^{2n},SG] \oplus [S^{2n+1},SG]$ are realised by normal invariants of pinch maps. This follows by composing the previous maps, using the composition formula~\cite[2.5--2.6]{madsen}.
\end{proof}

Combining Lemmas~\ref{lem:3.1} and~\ref{lem:3.2} with the composition formula, we can now show that all elements of $\operatorname{Im}(\tau_*)$ are realised by compositions of pinch maps.

\begin{theorem}\label{thm:3.3}
Every element of $\operatorname{Im}(\tau^*) \subset [V, SG]$ has a representative of the form $\varphi^{(\Sigma,g)} : V \mathbin{\#} \Sigma \xrightarrow{h_\Sigma} V \xrightarrow{P_g} V$.
\end{theorem}

\begin{proof}

Recall that for a pinch map $P_f : M \xrightarrow{q} M \vee S^n \xrightarrow{\mathrm{Id} \vee f} M \vee M \xrightarrow{\Delta} M$, the induced map satisfies $(P_f)^{*}(x) = x + c^{*}f^{*}x$. Consider the composition $V \xrightarrow{P_{(\ell\circ\alpha)}} V \xrightarrow{P_{(s\circ\beta)}} V$. A direct computation using Lemmas~\ref{lem:3.1} and~\ref{lem:3.2} gives:
\begin{align*}
&i^{*}\,\eta^{t}\bigl(
    P_{(s\circ\beta)}\circ P_{(\ell\circ\alpha)},\;
    (\mathrm{Id}\mathbin{\#} b_{s}\circ\mu_{\beta})
    \circ
    (\mathrm{Id}\mathbin{\#} b_{\ell}\circ\mu_{\alpha})
  \bigr) \\[4pt]
&\quad =\;
  i^{*}\,\eta^{t} \bigl(
    P_{(s\circ\beta)},\;
    \mathrm{Id}\mathbin{\#} b_{s}\circ b_{\beta}
  \bigr)
  \;+\;
  i^{*} \bigl(P_{(s\circ\beta)}^{*}\bigr)^{-1}
  \,\eta^{t} \bigl(
    P_{(\ell\circ\alpha)},\;
    \mathrm{Id}\mathbin{\#} b_{\ell}\circ b_{\alpha}
  \bigr),\\
&\quad =\;  
  i^{*}\,\eta^{t} \bigl(
    P_{(s\circ\beta)},\;
    \mathrm{Id}\mathbin{\#} b_{s}\circ b_{\beta}
  \bigr) \;+i^{*}\eta^{t}\!\left(P_{(\ell\circ\alpha)},\;
    \mathrm{Id}\mathbin{\#}b_{\ell}\circ b_{\alpha}\right)
  +\, i^{*}c^{*}(s\circ\beta)^{*}\,
  \eta^{t}\!\left(P_{(\ell\circ\alpha)},\;
    \mathrm{Id}\mathbin{\#}b_{\ell}\circ b_{\alpha}\right).\\
&\quad=\;  
  i^{*}\,\eta^{t} \bigl(
    P_{(s\circ\beta)},\;
    \mathrm{Id}\mathbin{\#} b_{s}\circ b_{\beta}
  \bigr) \;+i^{*}\eta^{t}\!\left(P_{(\ell\circ\alpha)},\;
    \mathrm{Id}\mathbin{\#}b_{\ell}\circ b_{\alpha}\right). \quad (\text{since } i^*c^* = 0)\\
&\quad=\;
  \bigl(\pm\{\beta\},\,\pm\{\alpha\}\bigr)
  \;\;\text{or}\;\;
  \bigl(\pm\{\beta\},\,\pm\{\alpha\}+\{\beta\circ\eta\}\bigr). \tag{3.16}\label{eq:3.16}
\end{align*}

Given any $(r,s) \in [S^{2n},SG] \oplus [S^{2n+1},SG]$, we can choose $\beta$, $\alpha$ such that $r = \pm\{\beta\}$ and either $s = \pm\{\alpha\}$ or $s - \{\beta \circ \eta\} = \pm\{\alpha\}$~\cite[Lemma~5.3]{madsen}. Thus compositions of pinch maps realise all elements of $[S^{2n},SG] \oplus [S^{2n+1},SG]$. Let $D$ denote the set of pinch maps $P_{(a,b)}$ corresponding to $(a,b) \in [S^{2n},G/O] \oplus [S^{2n+1},G/O]$.
We have the following exact sequence:
\begin{equation}\label{eq:3.17}\tag{3.17}
0 \to [S^{4n+1},G/O] \xrightarrow{c^*} [V,G/O] \xrightarrow{i^*} [S^{2n},G/O] \oplus [S^{2n+1},G/O] \to 0
\end{equation}
Let $C = \operatorname{Im}(c^*)$. Any element of $[V,SG]$ has the form $c^*(x) + d$. Every element on the left comes from $\Theta_{4n+1}$, so $c^*(x) = c^*(\eta_s(\Sigma_x)) = \eta(h_{\Sigma_x})$. By~\cite{madsen}, $\eta^t(P_{(a,b)} \circ h_{\Sigma_x}) = \eta^t(P_{(\alpha,\beta)}) + c^*(x)$, so the normal invariant map is surjective onto the image of $[V,SG]$ in $[V,G/O]$.
\end{proof}

\subsubsection*{3.1.B.\ Normal invariants of pinch maps: CW representation}

We now reformulate these results in terms of the CW structure of $V$, which will be needed for the rigidity analysis in \S\,4.

\begin{theorem}\label{thm:3.4}
Every element of $\operatorname{Im}(\tau^*) \subset [V, SG]$ has an inverse of the form $\varphi^{(\Sigma,g)} : V \mathbin{\#} \Sigma \xrightarrow{h_\Sigma} V \xrightarrow{P_g} V$, where $g$ factors through the lower skeleton of $V$.
\end{theorem}

\begin{proof}
We already know $P_{(\ell\circ\alpha)} \simeq P_{\alpha}$ up to the homotopy equivalence $q : V \to V_{\mathrm{cw}}$~\cite[Lemma~4.1]{nomura}.
So, $i^{*}\eta^{t}(P_{\alpha})=i^{*}\eta^{t}(P_{(\ell\circ\alpha)})
=(\pm\{\alpha\},0)$. For the other map, $P_{(s\circ\beta)}\simeq P_{\beta}$ if $q\circ s\simeq\pm i\circ j$. Then, $i^{*}\eta^{t}(P_{\beta})=(\pm\{\beta\},0) \ \
\text{or,} \ \ (\pm\{\beta\},\{\beta\circ\eta\})$. Otherwise, $q\circ s\simeq\pm i\circ (j + \eta) $, and we have the following diagram where upper triangle commutes and lower square commute up to homotopy:\\
\begin{equation}\label{eq:3.18}\tag{3.18}
\begin{tikzcd}[row sep=1.6em, column sep=3em]
 & S^{2n} \arrow[dr, "\ell"] & & \\
S^{4n+1}
  \arrow[r, "\beta"']
  \arrow[ur, "\eta\circ\beta"]
& S^{2n+1}
    \arrow[r, "s"]
    \arrow[d, "j+\eta"']
& V
    \arrow[d, "q"', "\simeq" right]
& \\
 & S^{2n+1}\vee S^{2n}
    \arrow[r, "i"']
    & V_{cw} 
\end{tikzcd}
\end{equation}
This gives the following relation between maps:
$q_{*}(s\circ\beta+\ell\circ\eta\circ\beta)
  = q\circ s\circ\beta + q\circ\ell\circ\eta\circ\beta
  = q\circ s\circ\beta + i\circ\eta\circ\beta
  = i\circ(j+\eta)\circ\beta + i\circ\eta\circ\beta
  = i\circ j\circ\beta + i\circ\eta\circ\beta + i\circ\eta\circ\beta
  = i\circ j\circ\beta.$\\
Then, $P_{\beta} = P_{(s\circ\beta+\ell\circ\eta\circ\beta)}
= P_{(s\circ\beta)}\circ P_{(\ell\circ\eta\circ\beta)}$ up to the homotopy equivalence between $V$ and $V_{cw}$. So, $i_{*}\eta^{t}(P_{\beta})
= i^{*}\eta^{t}(P_{(s\circ\beta)}) + i^{*}\eta^{t}(P_{(\ell\circ\eta\circ\beta)}) = \bigl(\pm\{\beta\},\,\pm\{\eta\circ\beta\}\bigr) \quad\text{or,}\quad
\bigl(\pm\{\beta\},\,0\bigr) $, which is the same in both cases. By the same argument as for Theorem~\ref{thm:3.3} we obtain the result.
\end{proof}

\subsection{Calculating Normal Invariants for $V_{2n+1,2}$}

We have the fibration
\[
S^{2n-1} \xrightarrow{l} V_{2n+1,2} \xrightarrow{\pi} S^{2n}
\]
We also have a CW structure $V_{2n+1,2} \simeq (S^{2n} \cup_2 S^{2n+1}) \cup_{\rho} e^{4n+1}$ with $\Sigma^2\rho \simeq 0$~\cite[Theorem~7.10]{james}.

There are two exact sequences depending on $b \pmod{4}$:
\[
1 \to \operatorname{Tor}\,\pi_{2b-3}(V_{b,2}) \to \mathcal{E}(V_{b,2}) \to \mathbb{Z}_2 \to 1,
\quad \text{for } b \neq 3,5,9,\ b \equiv 1\pmod{4},
\]
and
\[
1 \to \operatorname{Tor}\,\pi_{2b-3}(V_{b,2}) \xrightarrow{\cong} \mathcal{E}(V_{b,2}) \to 1,
\quad \text{for } b \neq 3,5,9,\ b \equiv 3\pmod{4},
\]
where $b = 2n+1$. In both cases the first map is given by pinch maps. Here $\mathbb{Z}_2 = \mathcal{E}(S^{2n-1} \cup_2 e^{2n}) = \{1, 1 + i\,\eta_P\}$, so an element $x$ of $\mathcal{E}(V)$ corresponding to $\mathrm{Id} + i\,\eta_P$ is an extension to $V$~\cite[Theorem~2.4]{nomura}. We compute the normal invariants of the pinch maps; the normal invariant of the $\mathbb{Z}_2$-extension is discussed in Remark~\ref{rem:4.3}.

\subsubsection*{3.2.A.\ Normal invariants of pinch maps}
There is a commutative diagram of exact sequences~\cite[Corollary~5.9]{madsen}:
\begin{equation}\label{eq:3.19}\tag{3.19}
\begin{tikzcd}[row sep=2em, column sep=large]
\{S^{2n}, S^0\}
  \arrow[r, "q^*"]
&
\{S^{2n-1} \cup_{2} e^{2n}, S^0\}
  \arrow[r, "i^*"]
&
\{S^{2n-1}, S^0\}
\\
\pi_{4n-1}^s(S^{2n-1})
  \arrow[r, "i_*"]
  \arrow[u, "\cong\, D"']
&
\pi_{4n-1}^s(S^{2n-1} \cup_{2} e^{2n})
  \arrow[r, "q_*"]
  \arrow[u, "\cong\, D"']
&
\pi_{4n-1}^s(S^{2n})
  \arrow[u, "\cong\, D"']
\\
&
\pi_{4n-1}(S^{2n-1} \cup_{2} e^{2n})
  \arrow[u, "\Sigma^K", "onto"']
&
\end{tikzcd}
\end{equation}

\begin{theorem}\label{thm:3.5}
Given any $\{\alpha\} \in \{S^{2n},S^0\}$ where $\alpha \in \pi_{4n-1}(S^{2n-1})$, we have $q^*\{\alpha\} = i^*\tilde{\eta}^t(p_{(\pm i_*\alpha)},\mathrm{Id} \# b_l \circ b_\alpha)$. Hence every element in the image of $q^*$ is realised by normal invariants of certain pinch maps.
\end{theorem}

\begin{proof}
Let $S^{4n-1} \xrightarrow{\alpha} S^{2n-1}$ be a map and let $p_{\alpha}$ denote the corresponding pinch map, giving the following commutative diagram:
\begin{equation}\label{eq:3.20}\tag{3.20}
\begin{tikzcd}
S^{4n-1} \times \mathbb{R}^k \arrow[r, "\alpha \times \mathrm{Id}"] \arrow[d, "\cong", "\psi"{left}]
  & S^{2n-1} \times \mathbb{R}^k \arrow[r, "l \times \mathrm{Id}"] \arrow[d, "\varphi", "\cong"{left}] 
  & V \times \mathbb{R}^k \arrow[d, "\widetilde{\varphi}", "\cong"'] \\
\nu_{S^{4n-1}} \arrow[r, "b_{\alpha}"'] \arrow[d] 
  & l^{*}\nu_{V} \arrow[r, "b_{l}"'] \arrow[d] 
  & \nu_{V} \arrow[d] \\
S^{4n-1} \arrow[r, "\alpha"] 
  & S^{2n-1} \arrow[r, "l"'] 
  & V
\end{tikzcd}
\end{equation}
Taking Thom-space level maps on the upper squares gives (here $D = D_{4n-1+2k}$):
\begin{equation}\label{eq:3.21}\tag{3.21}
\begin{tikzcd}[column sep=large]
S^{4n-1+k}
  \arrow[r]
  \arrow[d, "D", dotted] 
& \operatorname{Th}(\nu_{S^{4n-1}}) 
  \arrow[r, "\operatorname{Th}(b_\alpha)"] 
  \arrow[d, "D",dotted] 
& \operatorname{Th}(l^* \nu_V) 
  \arrow[r, "\operatorname{Th}(b_l)"] 
  \arrow[d, "D",dotted] 
& \operatorname{Th}(\nu_V) 
  \arrow[d, "D",dotted] \\
S^k 
  & \Sigma^k S_+^{4n-1} 
      \arrow[l, "\Sigma^k c_{S^0}"'] 
  & \Sigma^k \operatorname{Th}(\nu(s)) 
      \arrow[l, "D(\operatorname{Th}(b_\alpha))"'] 
  & \Sigma^k V_+ 
      \arrow[l, "\Sigma^k l_+^1"']
\end{tikzcd}
\end{equation}
The map $i^*\tilde{\eta}^t(p_\alpha, \mathrm{Id} \# b_l \circ b_\alpha)$ is the bottom composite of the following commutative diagram:
\begin{equation}\label{eq:3.22}\tag{3.22}
\begin{tikzcd}[column sep=3.5em,row sep=3.0em]
\Sigma^{k} V
  \arrow[r, "\Sigma^{k} l^{!}"]
 &
\Sigma^{k} \operatorname{Th}(\nu_{l})
  \arrow[r, "D(\operatorname{Th}(b_{\alpha}))"]
  \arrow[d, "\simeq"{left}, "D(h)"]
&
\Sigma^{k} S^{4n-1}_{+}
  \arrow[r,"\Sigma^k C_{s_0}"]
  \arrow[d,"D(m)"{left},"\simeq"{right}]
&
S^{k}
\\
\Sigma^{k} (S^{2n-1} \cup_2 e^{2n})
  \arrow[u,"i"]
  \arrow[r, "q"']
&
S^{2n+k} \;\vee\; S^{4n-1+k}
  \arrow[r, "{\{\alpha\} \vee \mathrm{Id}}"']
&
S^{k} \vee S^{4n-1+k}
  \arrow[ru, "\tilde{c}", bend right=15]
\end{tikzcd}
\end{equation}
Here we used that the dual of a wedge is the wedge of the duals~\cite[Theorem~6.5, Corollary~6.7]{funcspaces}. The lower-left map is the quotient map $q$ because $[S^{2n-1+k} \cup_2 e^{2n+k}, S^{2n+k}] \cong \mathbb{Z}_2\{q\}$, and $\Sigma^k\iota^!$ is nonzero in $H^{2n+k}(-;\mathbb{Z}_2)$, forcing the cohomology map to be $q$. Hence $i^*\tilde{\eta}^t(p_{(i_*\alpha)}, \mathrm{Id} \# b_l \circ b_\alpha) = \pm q^*\{\alpha\}$.

\end{proof}

Let $S^{4n-1} \xrightarrow{\beta} S^{2n-1} \cup_2 e^{2n}$ be a map with $q^*\{\beta\} = \hat{\beta}$. By~\cite[Lemmas~6.5,~7.4]{crowley:kervaire}, $i^*j^*\eta^t(P_{\beta \circ j}, \mathrm{Id} \# \beta) = [1]*i^*j^*\hat{\eta}^t(\beta \circ j, \bar{\beta})$. However, we cannot directly apply these techniques since $S^{2n-1} \cup_2 e^{2n}$ is not a submanifold of $V$. We proceed as follows.

Let $\beta : S^{4n-1} \to V$ be a map factoring through $S^{2n-1} \cup_2 e^{2n} \xrightarrow{j} V$. We have the following commutative diagrams:
\[
    \begin{minipage}{0.35\textwidth}
        \centering
        \begin{tikzcd}[row sep=large, column sep=small]
            S^{4n-1} \times \mathbb{R}^k \arrow[r, "\beta \times \mathrm{Id}"] \arrow[d, "\cong"'] 
            & V \times \mathbb{R}^k \arrow[d, "\cong"] \\
            \nu_{S^{4n-1}} \arrow[r, "\bar{\beta}"] \arrow[d] 
            & \nu_V \arrow[d] \\
            S^{4n-1} \arrow[r, "\beta"] 
            & V
        \end{tikzcd}
        \textit{Left:}~$(3.23)$
    \end{minipage}
    \quad \vrule \quad
    \begin{minipage}{0.55\textwidth}
        \centering
        \begin{tikzcd}[row sep=large, column sep=small]
            \Sigma^k S^{4n-1}_+ \arrow[r, "\Sigma^k \beta_+"] \arrow[d, "\cong"'] 
            & \Sigma^k V_+ \arrow[d, "\cong"] \\
            \mathrm{Th}(S^{4n-1} \times \mathbb{R}^k) \arrow[r, "T(\beta \times \mathrm{Id})"] \arrow[d, "\cong"'] 
            & \mathrm{Th}(V \times \mathbb{R}^k) \arrow[d, "\cong"] \\
            \mathrm{Th}(\nu_{S^{4n-1}}) \arrow[r, "T(\bar{\beta})"] 
            & \mathrm{Th}(\nu_V)
            \end{tikzcd}
            \textit{Right:}~$(3.24)$
    \end{minipage}
\]

In the right diagram, call the composite isomorphism on the left $\bar{\xi}$ and on the right $\xi$. This gives the following diagram and its Spanier--Whitehead dual:
\begin{equation}\label{eq:3.25}\tag{3.25}
\begin{tikzcd}[row sep=2em, column sep=3em]
& & S^{4n-1+k} \vee S^k \arrow[r, "\beta", dashed, bend left=20]
\arrow[r,"\Sigma^k \beta \vee Id"{below}]
\arrow[d, "\bar{\tau_+}","\simeq"{left}]
& (S^{2n-1+k} \cup_2 e^{2n+k}) \vee S^{4n-1+k} \vee S^k \arrow[d,"\tau_+","\simeq"{left}]\\
& & \Sigma^k S^{4n-1}_+ \arrow[r,"\Sigma^k \beta_+" ]  \arrow[d, "\cong"', "\bar{\xi}"]
& \Sigma^k V_+ \arrow[d, "\cong", "\xi","\simeq"{left}] \\
S^{4n-1+k} \arrow[rruu,"\mathcal{L}"]
\arrow[rru, "\mathcal{\psi}", thick]  \arrow[rr, "\text{coll.}"{below}]
& 
& \mathrm{Th}(\nu_{S^{4n-1}}) \arrow[r,"T(\bar{\beta})"]
& \mathrm{Th}(\nu_V)
\end{tikzcd}
\end{equation}
\begin{center}
(Take dual $D = D_{4n-1+2k}$)
\end{center}
\begin{equation}\label{eq:3.26}\tag{3.26}
\begin{tikzcd}[row sep=1.8 em, column sep=large]
(S^{2n-1+k} \cup_2 e^{2n+k}) \vee S^{4n-1+k} \vee S^k 
    \arrow[r, "D\beta \vee \mathrm{Id}"] 
    \arrow[d, "\simeq"', "D(\bar{\tau}_+ \circ \bar{\xi})"] 
& S^{4n-1+k} \vee S^k 
    \arrow[d, "\simeq", "D(\tau_+ \circ \xi)"'] 
    \arrow[dr] \\
\Sigma^k V_+ 
    \arrow[r,"D(T(\bar{\beta})"] 
& \Sigma^k S^{4n-1}_+ 
    \arrow[r,"C_{s_0}"] 
& S^k
\end{tikzcd}
\end{equation}
Call the lower map $\widetilde{N}^t(\beta)$. We have the following sequence of adjunctions:
\begin{equation}\label{eq:3.27}\tag{3.27}
\begin{tikzcd}[row sep=0.5em, column sep=large]
{[\Sigma^k V_+, S^k]_*} \arrow[r, "\cong", "\text{adj}"{below}] 
& {[V_+, \Omega^k S^k]_*} \arrow[r,dotted] 
& {[V, (\Omega^k S^k)_0]_*} \arrow[r, "\cong", "\text{adj}"{below}] 
& {[\Sigma^k V, S^k]_*} \\
\widetilde{N}^t(\beta) \arrow[rr, mapsto] 
&& \hat{\eta}^t(\beta \circ j,\bar{\beta}) \arrow[r, mapsto] 
& \text{adj } \hat{\eta}^t(\beta \circ j,\bar{\beta})
\end{tikzcd}
\end{equation}
The dotted arrow is not a map; our element merely passes through it. By~\cite[Lemmas~6.5,~7.4]{crowley:kervaire}, $\eta^t(P_\beta, \mathrm{Id} \# \bar{\beta}) = \eta^t([V,\mathrm{id},\mathrm{Id}] + [S^{4n-1},\beta,\bar{\beta}]) = [1] * \hat{\eta}^t(\beta \circ j,\bar{\beta})$. We therefore want $\operatorname{adj}\hat{\eta}^t(\beta \circ j,\bar{\beta})$; more precisely, we compute $i^*j^*\operatorname{adj}\hat{\eta}^t(\beta \circ j,\bar{\beta})$, where $S^{2n-1} \xrightarrow{i} S^{2n-1} \cup_2 e^{2n} \xrightarrow{j} V$ are composition of inclusions. We first establish:

\begin{theorem}\label{thm:3.6}
\begin{equation}\label{eq:3.28}\tag{3.28}
\begin{tikzcd}[row sep=0.5em, column sep=1.7em]
{[\Sigma^k V \vee S^k,S^k]}_*  
  \arrow [r,"\tau_*","\equiv"{below}]
& {[\Sigma^k V_+, S^k]_*}
  \arrow[r, "\cong", "\text{adj}"'] 
& {[V_+, \Omega^k S^k]_*} 
  \arrow[r,dotted] 
& {[V, (\Omega^k S^k)_0]_*} 
  \arrow[r, "\cong", "\text{adj}"'] 
& {[\Sigma^k V, S^k]_*} \\
\{\beta\} \vee 0 
  \arrow[rrrr] 
&&&& \{\beta\}
\end{tikzcd}
\end{equation}
where $\{\beta\}\vee 0$ passes through $[V,(\Omega^k S^k)_0]$, and $\tau$ is the natural basepoint-preserving homotopy equivalence $\Sigma^k V_+ \simeq \Sigma^k V \vee S^k$.
\end{theorem}
Note that, using structure of the map $\tau$ the elements of $[\Sigma^k V \vee S^k,S^k]$ with second summand zero exactly corresponds to the maps in $[V_+, \Omega^k S^k]_*$ such that, it's restriction in $V$ falls in $[V, (\Omega^k S^k)_0]_*$. So, we only need to prove, $\{\beta\} \vee 0 $ goes to $\{\beta\}$.

\begin{lemma}\label{lem:3.7}
The following diagram commutes:
\[
\begin{tikzcd}[scale=0.82, transform shape, row sep=2.3em, column sep=1.5em]
{[V_+, (\Omega^N S^N)_0]_*} \arrow[r, "i_*"] \arrow[d, "\cong"]
& {[V_+, \Omega^N S^N]_*} \arrow[r, "\Sigma^N"] 
& {[\Sigma^N V_+, \Sigma^N \Omega^N S^N]_*} \arrow[r, "\epsilon_*"] \arrow[d, "\equiv", "\tau^*"']
& {[\Sigma^N V_+, S^N]} \arrow[d, "\equiv", "\tau^*"{left}] \\
{[V \vee S^0, (\Omega^N S^N)_0]_*} \arrow[r, "j_*"]
& {[V \vee S^0, \Omega^N S^N]_*} \arrow[r, "\Sigma^N"]
& {[\Sigma^N V \vee S^N, \Sigma^N \Omega^N S^N]_*} \arrow[r, "\epsilon_*"]
& {[\Sigma^N V \vee S^N, S^N]}
%
\arrow[from=1-2, to=1-4, bend left=20, "\text{adj}", "\cong"{below}]
\arrow[from=2-2, to=2-4, bend right=20, "\text{adj}", "\cong"'{below}]
\end{tikzcd}
\tag{3.29}\label{eq:3.29}
\]
\end{lemma}
\begin{proof}
We know the Loop-Suspension adjunction isomorphism is given by the unit-counit map~\cite{may}, so the adjoint isomorphism commutes with lower compositions. The second square commutes automatically. we wish to show the first square commutes. Starting with $f:V_+ \to (\Omega^nS^n)_0$ on the top-left, which maps down to $f:V \vee S^0 \to (\Omega^nS^n)_0$, we write $f=g \vee \{0\}$. This goes to $\Sigma^n V \vee S^N \xrightarrow{\Sigma^n g \vee 0} \Sigma^n\Omega^nS^n$, while in the top row the map goes to $\Sigma^n f$. It suffices to show the following diagram commutes:
\begin{equation}\label{eq:3.30}\tag{3.30}
\begin{tikzcd}
\Sigma^n (V_+,+) \arrow[r," \Sigma^n f "] \arrow[d,"\tau","\simeq"{left}] & \Sigma^n (\Omega^n S^n) \\
\Sigma^n(V,v) \vee S^n \arrow{ur}[swap]{\Sigma^n g \vee \{0\}}
\end{tikzcd}
\end{equation}
Let $(V, v_0)$, $(V_+, +)$, $(V_+, v_0)$, $(Y, y)$ be based spaces. $f: V_+ \to Y$ be a map and $f|_V = g:(V,v_0) \to (Y,y). f(\{v_0,+\}=\{y\}). \text{Then} f=q \circ g_+: (V_+,+) \xrightarrow{g_+} (Y_+,+) \xrightarrow{q} (Y,y)$. Notice, $q: (Y_+, y) \to (Y, y)$ also.

We have the following diagram (which need not commute a priori):
\begin{equation}\label{eq:3.31}\tag{3.31}
\begin{tikzcd}
\Sigma(V_+, +) \arrow[r,"\Sigma g_+"] & \Sigma(Y_+, +) \arrow[r,"\Sigma q"] & \Sigma(Y, y) \\
\Sigma(V_+, v_0) \arrow[u,"\eta_V","\simeq"{right}] \arrow[r,"\Sigma g_+"] & \Sigma(Y_+, y) \arrow[u,"\eta_Y","\simeq"{right}] \arrow[ur,"\Sigma q",bend right=15] \\
\Sigma(V, v_0) \vee S^1 \arrow[uu,bend left = 80,"\tau","\simeq"{right}]\arrow[u,"\cong"] \arrow[r,"\Sigma g \vee Id"] & \Sigma(Y, y) \vee S^1 \arrow[u,"\cong"] \arrow[uur,bend right=30,"q|_{S^1}"]
\end{tikzcd}
\end{equation}
(In~\eqref{eq:3.31}: the upper and middle composition maps both equal $\Sigma f$; the lower composition map is $\Sigma g \vee \{0\}$.)

\paragraph*{Step~1.} \textit{The top-left square, bottom square, and bottom triangle commute.}

Pf: $(X, x) \xrightarrow{F_1, F_2, F_3} \Sigma(X_+, x), \Sigma X_+, \Sigma(X_+, +) \text{ are functorial}$. The homotopy equivalence between $\Sigma(X_+, +)$ and $\Sigma(X_+, x)$ is a natural transformation, so the top-left square commutes. The bottom square and triangle commute trivially.

\paragraph*{Step~2.} \textit{The upper triangle commutes.}

The following commutes from definition.
\begin{equation}\label{eq:3.32}\tag{3.32}
\begin{tikzcd}[row sep=large, column sep=large]
\Sigma(Y_+,+)  \arrow[d, "\Sigma q"] 
& SY_+ \arrow[l, "\simeq"{above},"{C|_{\{+\}\times I}}"] \arrow[r, "\simeq"{below},"{C|_{\{y\}\times I}}"] \arrow[d, "Sq"] 
& \Sigma(Y_+,y) \arrow[d, "\Sigma q"] \\
\Sigma(Y,y)  
& SY \arrow[l, "\simeq"{above},"C'|_{\{y\}\times I}"] \arrow[r, "\simeq"{below},"C'|_{\{y\}\times I}"] 
& \Sigma(Y,y)
\end{tikzcd}
\end{equation}
All maps are pointed with respect to the new basepoints. Choose vertex $y_1 = [(y_0,1)]$ as the basepoint of $SY$ and $SY_+$. Since $C|_{\{y\}\times I}$ is a pointed homotopy equivalence (assuming $SY_+, SY$ are well-pointed), we obtain a pointed homotopy inverse $\tilde{c}$. Similarly $C'|_{\{y\}\times I}$ is a pointed homotopy equivalence, giving inverse $\tilde{\tilde{c}}$. This yields the following pointed homotopy commutative diagram:

\medskip
\begin{center}
\begin{tabular}{@{}c@{}c@{}c@{}}
\begin{tikzcd}[row sep=2em, column sep=3.2em]
\Sigma(Y_+,+)  \arrow[d, "\Sigma q"] 
& SY_+ \arrow[l, "\simeq"{above},"{C|_{\{+\}\times I}}"]  \arrow[d, "Sq"] 
& \Sigma(Y_+,y) \ : \ \eta_{_Y} \arrow[d, "\Sigma q"] \arrow[l,"\tilde{c}","\simeq"{above}] \\
\Sigma(Y,y)  
& SY \arrow[l, "\simeq"{above},"C'|_{\{y\}\times I}"]  
& \Sigma(Y,y) \arrow[l,"\tilde{\tilde{c}}","\simeq"{above}] \arrow[ll,"Id",bend left = 35]
\end{tikzcd}
& \quad$\Rightarrow$\quad &
\begin{tikzcd}[row sep=2em, column sep=2em]
\Sigma(Y_+,+) \arrow[d, "\eta_Y", "\simeq"{left}] \arrow[r, "\Sigma q"] & \Sigma(Y,y) \\
\Sigma(Y_+,y) \arrow[ur, "\Sigma q"]
\end{tikzcd}
\\[4pt]
$(3.33)$ & & $(3.34)$
\end{tabular}
\end{center}
\medskip
Since $Y = \Omega^N S^N$ is a countable CW complex, $(Y,y)$ is well-pointed, as are $(Y_+,+)$ and $(Y_+,y)$. Since $Y_+$ is a CW complex, $SY_+$ is well-pointed. Hence the upper triangle commutes.

Together these show that diagram~\eqref{eq:3.31} commutes, from which the following pair of commutative diagrams establishes that the left square in diagram~\eqref{eq:3.29} also commutes (see~\eqref{eq:3.30}) :
\medskip
\begin{center}
\begin{tabular}{@{}c@{}c@{}c@{}}
\begin{tikzcd}[row sep=3.2em, column sep=3.2em]
    \Sigma(V_+,+) \arrow[d,"\tau","\simeq"{left}] \arrow[r,"\Sigma f"] & \Sigma(Y,y)\\
    \Sigma (V,v) \vee S^1 \arrow[ru,"\Sigma g \vee \{0\}"]
\end{tikzcd}
& \quad$\Rightarrow$\quad &
\begin{tikzcd}[row sep=3.2em, column sep=3.2em]
    \Sigma^N(V_+,+) \arrow[d,"\tau","\simeq"{left}] \arrow[r,"\Sigma^N f"] & \Sigma^N(Y,y)\\
    \Sigma^N (V,v) \vee S^N \arrow[ru,"\Sigma^N g \vee \{0\}"]
\end{tikzcd}
\\[1pt]
\multicolumn{3}{r}{$(3.35)$}
\end{tabular}
\end{center}
\end{proof}
\vspace{-3em}

\begin{proof}[Proof of Theorem~\ref{thm:3.6}]
We have the following commutative diagram.
\[
\begin{tikzcd}[scale=0.6, transform shape]
        {[\Sigma^N V,S^n]^0_*} & {[V,(\Omega^N S^N)_0]_*} \arrow[l,"\text{\Large adj}","\text{\Large $ \cong $}"{above}] \arrow[d,"\text{\Large $ Id $}","\text{\Large $ \cong $}"{left}] & {[V,(\Omega^N S^N)_0]} \arrow[l,"\text{\Large $\cong$}"] & {[V_+,(\Omega^N S^N)_0]_*} \arrow[l,"\text{\Large$\text{\Large $ \cong $}$}"] \arrow[d,"\text{\Large $ \cong $}","\text{\Large $ Id $}"{left}] \arrow[r,"\text{\Large $ adj $}","\text{\Large $\cong$}"{below}] & {[\Sigma^N V_+,S^N]_*^0} \arrow[d,"\text{\Large $ \tau^* $}*","\text{\Large $ \equiv $}"{left}]\\
    &{[V,\Omega^N S^N)_0]_*} \arrow[lu,"\text{\Large $ adj $}","\text{\Large $\cong$}"{right}, bend left = 25] && {[V \vee S^0, (\Omega^N S^N)_0]_*} \arrow[ll,"\text{\Large $ i^* $}","\text{\Large $\cong$}"{above}] \arrow[r,"\text{\Large $adj$}","\text{\Large $\cong$}"{below}] & {[\Sigma^n V \vee S^N, S^N]_0^*}
\end{tikzcd} \tag{3.36}
\] 
The rightmost square commutes by Lemma~\ref{lem:3.7}. The leftmost square commutes trivially. The middle square commutes too: given $f: V_+ \to (\Omega^NS^N)_0$, we may assume $f|_V = g$ is also pointed (since $(\Omega^NS^N)_0$ is a path-connected $H$-space and $V$ is well-pointed). On the top row $[f]$ maps to $[g]$; the same holds on the bottom row. The lower sequence starting from the rightmost term is restriction, and the map in the theorem is the composition from bottom-right to top-left via the top row, which by commutativity is simply restriction. Hence the theorem is proved.
\end{proof}

\begin{remark}\label{rem:3.8}
Diagram~\eqref{eq:3.27} and Theorem~\ref{thm:3.6} give $\operatorname{adj}\hat{\eta}^t(\beta \circ j,\bar{\beta}) = ((\tau^{-1})^* \tilde{N}^t(\beta))|_{\Sigma^k V}$.
\end{remark}

We now compute the normal invariants of certain maps that, together with Theorem~\ref{thm:3.5}, account for all of $[S^{2n-1} \cup_2 e^{2n},SG]$.

\begin{theorem}\label{thm:3.9}
\begin{enumerate}[label=\textup{\arabic*)}]
\item Let $S^{4n-1} \xrightarrow{\beta} S^{2n-1} \cup_2 e^{2n}$ be a map with $q^*\{\beta\} = \hat{\beta}$. Then $i^*j^*\operatorname{adj}\hat{\eta}^t(\pm \beta \circ j,\overline{\pm \beta}) = D(\hat{\beta})$.
\item For any $S^{4n-1} \xrightarrow{\alpha} S^{2n-1}$, $q^*\{\alpha\} = i^*\operatorname{adj}\tilde{\eta}^t (P_{\pm i^*j^*\alpha}, \mathrm{Id}_{\nu_V} \mathbin{\#} b_l \circ b_\alpha)$.
\end{enumerate}

For any path-connected space $X$ there is a bijection $H\colon\{X,S^0\}_* \xrightarrow[\cong]{\operatorname{adj}} [X,(\Omega^N S^N)_0]_* \xrightarrow[\equiv]{[1]*} [X,SG]_*$. For $x \in \{X,S^0\}_*$ denote by $x_H$ the corresponding element in $[X,SG]_*$. Furthermore:
\begin{enumerate}[label=\textup{\arabic*)}]\setcounter{enumi}{2}
\item For any $S^{4n-1} \xrightarrow{\alpha} S^{2n-1}$, $q^*(\alpha_H) = j^*\eta^t(P_{\pm i^*j^*\alpha}, \mathrm{Id}_{\nu_V} \mathbin{\#} b_l \circ b_\alpha)$.
\item Let $S^{4n-1} \xrightarrow{\beta} S^{2n-1} \cup_2 e^{2n}$ be a map with $q^*\{\beta\} = \hat{\beta}$. Then $(i^*j^*)\eta^t([P_{\pm j^*\beta}, \mathrm{Id}_{\nu_V} \mathbin{\#} \bar{\beta}]) = D(\hat{\beta})_H$.
\end{enumerate}
\end{theorem}

\begin{proof}[Proof of (1)]
\phantom\qedhere
We have the following commutative diagram: 
\begin{equation}\label{eq:3.37}\tag{3.37}
\begin{tikzcd}[scale=0.85, transform shape, row sep=1.8em, column sep=1.8em]
&&(S^{2n-1+k} \cup_2 e^{2n+k}) \vee S^{4n-1+k} \vee S^k 
    \arrow[r, "D\{ \beta \} \vee \mathrm{Id}"] 
    \arrow[d, "\simeq"', "D(\bar{\tau}_+ \circ \bar{\xi})"] 
& S^{4n-1+k} \vee S^k 
    \arrow[d, "\simeq", "D(\tau_+ \circ \xi)"'] 
    \arrow[dr, bend left = 15,"\tilde{c}"{above}] \\
&&\Sigma^k V_+ 
    \arrow[r,"D(T(\bar{\beta})"] \arrow[d,"\tau","\simeq"{left}] 
& \Sigma^k S^{4n-1}_+ 
    \arrow[r,"C_{s_0}"] 
& S^k \\
&&{\Sigma^k V \vee S^k} \arrow[rru,"\tilde{\beta}", bend right = 10] \arrow[d,"\simeq"] \\
{S^{2n-1+k}} \arrow[r,"i"] & {S^{2n-1+k} \cup_2 e^{2n+k}} \arrow[r,"j"] \arrow[ruuu,"G",bend left = 30] & {\Sigma^k V_{cw} \vee S^k} \arrow[uuu,"L"{right}, "\simeq"{left}, bend left = 65] \arrow[rruu,bend right = 20,"\hat{\beta}"]
\end{tikzcd}
\end{equation}
$L$ induces a homotopy equivalence $S^{2n-1+k} \cup_2 e^{2n+k} \vee S^k \xrightarrow{L'} S^{2n-1+k} \cup_2 e^{2n+k} \vee S^k$, so $G = L \circ j$ is of the form $S^{2n-1+k} \cup_2 e^{2n+k} \xrightarrow{f \vee h} S^{2n-1+k} \cup_2 e^{2n+k} \vee S^k$, where $f$ is a homotopy equivalence (by analysing $H_{2n-1+k}(L')$) and $h$ is some map into $S^k$. Thus $f$ is either $\mathrm{Id}$ or $3\,\mathrm{Id}$~\cite[Lemma~2.2]{homotopy5mfld}. Moreover, $\tilde{c} \circ (D(\{\beta\}) \vee \mathrm{Id}) \circ G \circ i = \pm\, i^*D\{\beta\}$ (using $2\,\mathrm{Id} \circ i = 0$). By Remark~\ref{rem:3.8}, $(i^*j^*\operatorname{adj}\hat{\eta}^t(\beta)) = \hat{\beta} \circ j \circ i = \pm\, i^*D\{\beta\} = D(\hat{\beta})$. Hence we have proved (1).

To make this precise, we organise the relevant maps in the following commutative diagram.
\begin{equation}\label{eq:3.38}\tag{3.38}
\begin{tikzcd}[scale=0.82, transform shape, row sep=2em, column sep=large]
{[S^{2n},SG]} \arrow[r,"q^*"] & {[S^{2n-1} \cup_{2} e^{2n}, SG]} \arrow[r,"i^*"] & {[S^{2n-1},SG]}\\
\{S^{2n}, S^0\}
  \arrow[r, "q^*"] \arrow[u,"H","\equiv"{right}]
&
\{S^{2n-1} \cup_{2} e^{2n}, S^0\} \arrow[u,"H"]
  \arrow[r, "i^*"]
&
\{S^{2n-1}, S^0\} \arrow[u,"H","\equiv"{right}]
\\
\pi_{4n-1}^s(S^{2n-1})
  \arrow[r, "i_*"]
  \arrow[u, "\cong\, D"']
&
\pi_{4n-1}^s(S^{2n-1} \cup_{2} e^{2n})
  \arrow[r, "q_*"]
  \arrow[u, "\cong\, D"']
&
\pi_{4n-1}^s(S^{2n})
  \arrow[u, "\cong\, D"']
\\
&
\pi_{4n-1}(S^{2n-1} \cup_{2} e^{2n})
  \arrow[u, "\Sigma^K", "onto"']
&
\end{tikzcd}
\end{equation}
In the notation of the commutative diagram~\eqref{eq:3.38} above, this proves parts (1). And, (2) is the exact statement of \textbf{Theorem 3.5}.
\end{proof}

\begin{proof}[Proof of (3)]
$q^*(\alpha_H) = H(q^*\alpha) = [1]*\operatorname{adj}\,q^*\{\alpha\} = [1]*i^*\tilde{\eta}^t(\pm i_*\alpha, b_l \circ b_\alpha) = i^*([1]*\tilde{\eta}^t(\pm i_*\alpha, b_l \circ b_\alpha)) = i^*\eta^t(P_{\pm i^*\alpha}, \mathrm{Id}_{\nu_V} \mathbin{\#} b_l \circ b_\alpha)$. \phantom\qedhere
\end{proof}

\begin{proof}[Proof of (4)]
$(i^*j^*)\eta^t([P_{\pm j^*\beta}, \mathrm{Id}_{\nu_V} \mathbin{\#} \bar{\beta}]) = (i^*j^*)([1]*\hat{\eta}^t([\pm j^*\beta, \bar{\beta}])) = [1]*(i^*j^*\hat{\eta}^t([\pm j^*\beta, \bar{\beta}])) = [1]*\operatorname{adj}\,D(\hat{\beta}) = D(\hat{\beta})_H$.
\end{proof}

We now show that the normal invariants computed above jointly generate all of $[S^{2n-1} \cup_2 e^{2n}, SG]$.

\begin{theorem}\label{thm:3.10}
The map $\epsilon^t(V) \xrightarrow{\eta^t} [V,SG] \xrightarrow{i^*} [S^{2n-1} \cup_2 e^{2n},SG]$ is surjective. Passing to $G/O$, the map $\epsilon(V) \xrightarrow{\eta} [V,SG] \xrightarrow{i^*} [S^{2n-1} \cup_2 e^{2n},G/O]$ is surjective onto $\mathrm{Im}(\tau_*)$, where $\tau\colon G \to G/O$.
\end{theorem}

\begin{proof}
Let $\cdots \to A \xrightarrow{f} B \xrightarrow{g} C \to \cdots$ be an exact sequence of groups, and let $B_g$ be any set containing at least one preimage of every element of $Im(C)$ ($B_g$ is not necessarily a group). Then $B = f(A) + B_g$. Denote $[P_{\pm i^*j^*\alpha}, \mathrm{Id}_{\nu_V} \mathbin{\#} b_l \circ b_\alpha] = [\xi_\alpha, \bar{\xi}_\alpha]$ and $[P_{\pm j^*\beta}, \mathrm{Id}_{\nu_V} \mathbin{\#} \bar{\beta}] = [\xi_\beta, \bar{\xi}_\beta]$. For the exact sequence \[
\begin{tikzcd}
\pi_{4n-1}^s(S^{2n-1})
  \arrow[r, "i_*"]
  &
\pi_{4n-1}^s(S^{2n-1} \cup_{2} e^{2n})
  \arrow[r, "q_*"]
  &
\pi_{4n-1}^s(S^{2n}) \end{tikzcd}
\]
for every $\hat{\beta} \in Im(q^*)$ choose, some $\{\beta\} \in \pi_{4n-1}^s(S^{2n-1} \cup_{2} e^{2n}) \ \text{and} \ (B_g)_1 = \{\{\ \beta\} \ | \ \hat{\beta} \in  im(q^*)\}$. For the exact sequence, \[
\begin{tikzcd}
{[S^{2n},SG]} \arrow[r,"q^*"] & {[S^{2n-1} \cup_{2} e^{2n}, SG]} \arrow[r,"i^*"] & {[S^{2n-1},SG]}.
\end{tikzcd}
\] Set $(B_g)_2 = \{ j^*\eta^t([\xi_\beta, \bar{\xi}_\beta]) \mid \{\beta\} \in (B_g)_1 \}$ (Theorem~\ref{thm:3.9}). Both the image of $q^*$ and $(B_g)_1$ are realised by $j^*\eta^t$ (Theorem~\ref{thm:3.9}). It remains to show that their sum (which equals all of $[S^{2n-1} \cup_{2} e^{2n}, SG]$) is also realised by $j^*\eta^t$.

$j^*\eta^t([\xi_b,\bar{\xi}_b] \circ [\xi_a,\bar{\xi}_a]) = j^*\eta^t([\xi_b,\bar{\xi}_b]) + j^*(\xi_{-a})^{*}\eta^t([\xi_b,\bar{\xi}_b]) = j^*\eta^t([\xi_a,\bar{\xi}_a]) + j^*\eta^t([\xi_b,\bar{\xi}_b]) = q^*(\alpha_H) + j^*\eta^t([\xi_b,\bar{\xi}_b])$, since $j^*(\xi_{-a})^{*} = \mathrm{Id}$ when $\xi_a$ is a pinch map~\cite[Lemma~2.6]{madsen}. This generates all elements of $\mathrm{Im}(q^*) + (B_g)_2 = [S^{2n-1} \cup_2 e^{2n}, SG]$.

\end{proof}

We can now establish the analogue of Theorem~\ref{thm:3.4} for $V_{2n+1,2}$: every element of $\operatorname{Im}(\tau_*)$ admits an explicit inverse.

\begin{theorem}\label{thm:3.11}
Any element of $\mathrm{Im}\,\tau_* \subset [V,G/O]$ has an inverse of the form $\varphi^{(\Sigma, g)}\colon V \mathbin{\#} \Sigma \xrightarrow{h_{\Sigma}} V \xrightarrow{P(g)} V$, where $g$ factors through the lower skeleton.
\end{theorem}

\begin{proof}
  There is a short exact sequence $0 \to [S^{4n-1},SG] \xrightarrow{c^*} [V,SG] \xrightarrow{j^*} [S^{2n-1} \cup_2 e^{2n}, SG] \to 0$. For any $l \in [S^{2n-1} \cup_2 e^{2n}, SG]$, choose a pinch map $P_l$ with $i^* \eta^t(P_l) = l$. Consider also $0 \to [S^{4n-1},G/O] \xrightarrow{c^*} [V,G/O] \xrightarrow{j^*} [S^{2n-1} \cup_2 e^{2n}, G/O] \to 0$. Any $x \in [V,SG]$ has the form $c^*(a)+\eta^t(P_l)$, so $\xi^*(x) = c^*\xi^*(a) + \eta(P_l)$. Since $c^*\xi^*(a) \in [S^{4n-1},G/O]$, there exists a homotopy sphere $\Sigma_a$ with $\eta(\Sigma_a) = c^*\xi^*(a)$. Then $\eta(P_l \circ h_\Sigma) = c^*\xi^*(a) + \eta(P_l)$, so maps $[P_l \circ h_\Sigma] \in hS(V)$ realise all normal invariants  
\end{proof}

\begin{remark}
    (Theorem~\ref{thm:3.3}) and (Theorem~\ref{thm:3.11}) together gives an alternative proof of the following theorem, that is, any manifold tangentially homotopy equivalent to $V_{n,2}$ is almost diffeomorphic to $V_{n,2}$. 
\end{remark}
\section{Rigidity result for V}

From \S\,2.2 we use the following criterion: a homotopy equivalence $f\colon V \to V$ is homotopic to an almost diffeomorphism if and only if $\eta(f) \in c^*[S^m,G/O] \subseteq [V,G/O]$, where $c$ is the degree-one map $V \to S^m$. Equivalently, $j^*\eta(f) = 0$, where $j$ is the inclusion of the lower skeleton.

\subsection{Rigidity result for $V_{2n+1,2}$}

\begin{lemma}\label{lem:4.1}
For a pinch map $P_g$ with $g\colon S^{4n-1} \to S^{2n-1} \cup_2 e^{2n}$, $j^* \eta(P_g) = 0$ if and only if:
\begin{enumerate}[label=\textup{(\roman*)}]
\item $j^* \eta^t(P_g, \bar{P}_g) = 0$, when $2n \equiv 4, 6 \pmod{8}$.
\item $j^* \eta^t(P_g, \bar{P}_g) = q^*(J(\xi_{\epsilon}))$ or $0$, when $2n \equiv 0 \pmod{8}$; here $\langle \xi_{\epsilon_{2n}} \rangle \cong [S^{2n}, O] \xrightarrow{J} [S^{2n}, G] \xrightarrow[\cong]{q^*} [S^{2n-1} \cup_2 e^{2n}, G]$.
\end{enumerate}
\end{lemma}

\begin{proof}
Let $V = V_{2n+1, 2} \simeq \left( S^{2n-1} \cup_2 e^{2n} \right) \cup_{\rho} e^{4n-1}$, $\Sigma^2 \rho = 0$. We have a commutative diagram:
\begin{equation}\label{eq:4.1}\tag{4.1}
\begin{tikzcd}[scale=0.85, transform shape]
0 \arrow[r] & {[S^{4n-1}, O]} \arrow[r] \arrow[d, "J"] & {[V, O]} \arrow[r, "j^*"] \arrow[d, "J"] & {[S^{2n-1} \cup_2 e^{2n}, O]} \arrow[r] \arrow[d, "J"] & 0 \\
0 \arrow[r] & {[S^{4n-1}, G]} \arrow[r] \arrow[d, "\tau_*"] & {[V, G]} \arrow[r, "j^*"] \arrow[d, "\tau_*"] & {[S^{2n-1} \cup_2 e^{2n}, G]} \arrow[r] \arrow[d, "\tau_*"] & 0 \\
0 \arrow[r] & {[S^{4n-1}, G/O]} \arrow[r] & {[V, G/O]} \arrow[r, "j^*"] & {[S^{2n-1} \cup_2 e^{2n}, G/O]} \arrow[r] & 0
\end{tikzcd}
\end{equation}
\begin{align*}
j^* \eta(P_{\alpha}) = 0 &\iff j^* \tau_* (\eta^t(P_{\alpha}, \bar{P}_{\alpha})) = 0 \\
&\iff \tau_* j^* (\eta^t(P_{\alpha}, \bar{P}_{\alpha})) = 0 \\
&\iff j^* (\eta^t(P_{\alpha}, \bar{P}_{\alpha})) \in \mathrm{Im}\,J.
\end{align*}
We have the cofiber sequence $S^{2n-1} \xrightarrow{\times 2} S^{2n-1} \xrightarrow{i} S^{2n-1} \cup_2 e^{2n} \xrightarrow{q} S^{2n} \xrightarrow{\times 2} S^{2n}$, leading to the following exact commutative diagram:
\begin{equation}\label{eq:4.2}\tag{4.2}
\begin{tikzcd}[scale=0.82, transform shape]
{[S^{2n}, O]} \arrow[r, "\times 2"] \arrow[d, "J"] & {[S^{2n}, O]} \arrow[r, "q^*"] \arrow[d, "J"] & {[S^{2n-1} \cup_2 e^{2n}, O]} \arrow[r, "i^*"] \arrow[d, "J"] & {[S^{2n-1}, O]} \arrow[r, "\times 2"] \arrow[d, "J"] & {[S^{2n-1}, O]} \arrow[d, "J"] \\
{[S^{2n}, G]} \arrow[r] & {[S^{2n}, G]} \arrow[r, "q^*"] & {[S^{2n-1} \cup_2 e^{2n}, G]} \arrow[r, "i^*"] & {[S^{2n-1}, G]} \arrow[r] & {[S^{2n-1}, G]}
\end{tikzcd}
\end{equation}
\begin{enumerate}[label=\textup{\arabic*)}]
\item For $2n \equiv 2, 4, 6 \pmod{8}$, $[S^{2n}, O] = 0$, so $[S^{2n-1} \cup_2 e^{2n}, O] \xrightarrow{\cong} \ker(\times 2)$. There are three sub-cases:
\begin{enumerate}[label=\textup{(\alph*)}]
\item $2n \equiv 6 \pmod{8}$: $[S^{2n-1}, O] = 0$, so $[S^{2n-1} \cup_2 e^{2n}, O] = 0$.
\item $2n \equiv 4 \pmod{8}$: $\mathbb{Z} \cong \pi_3(O) \xrightarrow{\times 2} \pi_3(O) \cong \mathbb{Z}$, so $\ker(\times 2) = 0$ and $[S^{2n-1} \cup_2 e^{2n}, O] = 0$.
\item $2n \equiv 2 \pmod{8}$: $\mathbb{Z}_2 \cong \pi_1(O) \xrightarrow{\times 2} \pi_1(O) \cong \mathbb{Z}_2$, so $\ker(\times 2) = \mathbb{Z}_2$ and $[S^{2n-1} \cup_2 e^{2n}, O] \cong \mathbb{Z}_2 \cong [S^{2n-1}, O]$.
\end{enumerate}
(a) ,  (b) proves (i).
\item For $2n \equiv 0 \pmod{8}$: $\mathbb{Z}_2  \xrightarrow{\times 2} \mathbb{Z}_2 \to [S^{2n-1} \cup_2 e^{2n}, O] \to \mathbb{Z} \xrightarrow{\times 2} \mathbb{Z}$, giving $\mathbb{Z}_2 \cong [S^{2n}, O] \xrightarrow[q^*]{\cong} [S^{2n-1} \cup_2 e^{2n}, O]$. So, $Im(J) = Im \ (\mathbb{Z}_2 \langle \xi_{\epsilon_{2n}} \rangle \cong [S^{2n}, O] \xrightarrow{J}[S^{2n}, G]  \xrightarrow[q^*]{\cong} [S^{2n-1} \cup_2 e^{2n}, G])$. This proves (ii).
\end{enumerate}
\end{proof}


\begin{theorem}\label{thm:4.2}
\leavevmode\begin{enumerate}[label=\textup{\arabic*)}]
\item When $2n \equiv 4, 6 \pmod{8}$: $P_g$ is an almost diffeomorphism if and only if $j^*\eta(P_g) = 0$, if and only if $\{g\} = i^*\{\alpha\}$ and $\{\alpha\}$ is divisible by $2$.
\item When $2n \equiv 0 \pmod{8}$: $P_g$ is an almost diffeomorphism if and only if $j^*\eta(P_g) = 0$, if and only if $\{g\} = i^*\{\alpha\}$ and $HD\{\alpha\}$ or $HD\{\alpha\} - J(\xi)$ is divisible by $2$.
\item When $2n \equiv 2 \pmod{8}$: $j^*\eta(P_g) = 0$ implies $q^*\{g\} = \hat{\xi}_1$ or $0$, where $\hat{\xi}_1$ satisfies $HD(\hat{\xi}_1) = J(\xi_1)$. This is a necessary condition only.
\end{enumerate}
\end{theorem}

\begin{proof}[proof of (1)]
For $2n \equiv 4, 6 \pmod{8}$, take $g\colon S^{4n-1} \to S^{2n-1} \cup_2 e^{2n}$ with $q^*\{g\} = \hat{\beta} \neq 0$. Then $i^*(j^*\eta^t(P_g, \bar{P}_g)) = H(D(\hat{\beta})) \neq 0$, so $j^*\eta^t(P_g, \bar{P}_g) = 0$ forces $\{g\} = i^*\{\alpha\}$, i.e.\ $g\colon S^{4n-1} \to S^{2n-1} \hookrightarrow S^{2n-1} \cup_2 e^{2n}$. Then $j^*\eta^t(P_g, \bar{P}_g) = \pm q^*(HD\{\alpha\})$, so $j^*\eta^t(P_g, \bar{P}_g) = 0 \Rightarrow q^*(HD\{\alpha\}) = 0 \Rightarrow i_*\{\alpha\} = 0$, i.e.\ $2 \mid \{\alpha\}$.Going backwards, this condition is also sufficient. \phantom\qedhere
\end{proof}
\begin{proof}[proof of (2)]
     For $2n \equiv 0 \pmod{8}$, a similar computation gives $q^*HD\{\alpha\} = 0$ or $q^*J(\xi)$, i.e.\ $HD\{\alpha\}$ or $HD\{\alpha\} - J(\xi)$ is divisible by $2$, as necessary and sufficient condition for $j^*\eta(P_g) = 0$.

\end{proof}

\begin{remark}\label{rem:4.3}
When $2n+1 \equiv 1, 5 \pmod{8}$, elements of $\epsilon(V)$ have the form $\hat{\eta} \circ P_g$, where $\hat{\eta}$ is any extension of $\mathrm{Id} + i\eta p \in \epsilon(S^{2n-1} \cup e^{2n})$. Lemma~\ref{lem:4.1} gives a necessary and sufficient condition for $P_g$ to be an almost diffeomorphism; an analogous condition for $\hat{\eta} \circ P_g$ requires computing $\eta(\hat{\eta})$. For $2n+1 \equiv 7 \pmod{8}$, all elements of $\epsilon(V)$ are pinch maps $P_g$, so we obtain an if-and-only-if condition for the full group $\epsilon(V)$ in that case. ~\cite[Theorem~2.4]{nomura}.
\end{remark}

\subsection{Rigidity results for $V_{2n+2,2}$}

\begin{theorem}\label{thm:4.4}
$P_g$ is homotopic to an almost diffeomorphism if and only if $g = \alpha \vee \beta$ and $\{\alpha\}, \{\beta\} \in \mathrm{Im}\,J$.
\end{theorem}

\begin{proof}
Let $V = V_{2n+2, 2} \simeq (S^{2n} \vee S^{2n+1}) \cup_{\rho} e^{4n+1}$. Every homotopy equivalence has the form $P_g \circ \phi$, where $\phi$ is a diffeomorphism~\cite[Theorem~2.4]{nomura}, and $g\colon S^{4n+1} \xrightarrow{\alpha \vee \beta} S^{2n} \vee S^{2n+1}$. In the exact sequence for $\epsilon(V)$, Whitehead products are quotiented, so $P_g = P_{\alpha} \circ P_{\beta}$. We have the following commutative diagram:
\begin{equation}\label{eq:4.3}\tag{4.3}
\begin{tikzcd}[scale=0.85, transform shape]
& {[V, SO]} \arrow[r, "i^*"] \arrow[d, "J"] & {[S^{2n}, SO] \oplus [S^{2n+1}, SO]} \arrow[d, "J"] \\
\epsilon^t(V) \arrow[r, "\eta^t"] \arrow[d, "F"] & {[V, SG]} \arrow[r, "i^*"] \arrow[d, "\tau_*"] & {[S^{2n}, SG] \oplus [S^{2n+1}, SG]} \arrow[d, "\tau_*"] \\
\epsilon(V) \arrow[r, "\eta"] & {[V, G/O]} \arrow[r, "i^*"] & {[S^{2n}, G/O] \oplus [S^{2n+1}, G/O]}
\end{tikzcd}
\end{equation}
\begin{align*}
i^* \eta(P_g) &= i^* \eta(P_{\alpha}) + i^* \eta(P_{\beta}) \\
&= \tau_* (i^* \eta^t(P_{\alpha}, \bar{P}_{\alpha})) + \tau_* (i^* \eta^t(P_{\beta}, \bar{P}_{\beta})) \\
&= \tau_*(0, \{\pm \alpha\}) + \tau_*(\{\pm \beta\}, \pm \{\beta \circ \eta\}) \quad\text{or}\quad \tau_*(0, \pm \alpha) + \tau_*(\pm \beta, 0) \\
&= \tau_*(\{\pm \beta\}, \{\pm \alpha \pm \beta \circ \eta\}) \quad\text{or}\quad \tau_*(\{\pm \beta\}, \{\pm \alpha\}).
\end{align*}
For the first option, $i^* \eta(P_g) = 0 \iff \{\beta\} \in \mathrm{Im}\,J_{2n}$ and $\{\alpha \pm (\beta \circ \eta) \} \in \mathrm{Im}\,J_{2n+1}$. Since $\mathrm{Im}\,J_{2n} = 0$ for $2n = 2, 4, 6$, this reduces to $\{\beta\} = 0$ and $\{\alpha\} \in \mathrm{Im}\,J_{2n+1}$, which coincides with the second option's condition. For $2n \equiv 0 \pmod{8}$, $\mathrm{Im}\,J_{2n} = \mathbb{Z}_2\{\eta x_{2k}\}$ and $\mathrm{Im}\,J_{2n+1} = \mathbb{Z}_2\{\eta^2 x_{2k}\}$~\cite[Theorem~1.1.13]{ravenel}, so again both options give $\{\beta\} \in \mathrm{Im}\,J_{2n}$ and $\{\alpha\} \in \mathrm{Im}\,J_{2n+1}$.
\end{proof}
\vspace{-2em}

\begin{remark}\label{rem:4.5}
Since every homotopy equivalence of $V$ is homotopic to $P_g \circ \phi$ for some diffeomorphism $\phi$, Theorem~\ref{thm:4.4} completely determines which self-homotopy equivalences of $V$ are homotopic to almost diffeomorphisms.
\end{remark}

\section{Manifolds tangentially homotopy equivalent to $V_{n,2} \times S^k$, where $ k = 3,5 \ or \ 7\leq k \ ; \ k \neq 2^i-2$, $Dim (V_{n,2} \times S^k) \neq 2^i-2$ .}
Throughout this section, we assume $k \neq 2^i-2$, $Dim (V_{n,2} \times S^k) \neq 2^i-2$.

Let $M$ be a smooth, closed manifold of dimension $\geq 5$. We have a surgery exact sequence of sets:
$L_{n+1}(e) \to hS(M) \xrightarrow{\eta} [M,G/O] \to L_{n}(e)$. Suppose for every element $x \in Im(\eta) \subset [M,G/O]$, we have  $[b_x: N_x \xrightarrow{\simeq} M] \in hS(M)$, s.t. $\eta (b_x) = x$. Then, take any smooth manifold $N \overset{q}{\simeq} M$. Let, $\eta(q) = y$. We have, $[b_y: N_y \xrightarrow{\simeq} M] \in hS(M)$, s.t. $\eta (b_y) = y$. Then, by theorem of Novikov, $N$ is almost diffeomorphic $N_y$ (by a $bp$ element). So, just finding out one inverses for each element in $Im(\eta)$ tells us how, the manifolds homotopy equivalent to $M$ looks up to almost diffeomorphism. And, Doing the same thing for all elements in $Im(\eta) \cap Im(\tau_*)$, tells how manifolds tangentially homotopy equivalent to $M$ looks up to almost diffeomorphism. So, almost diffeomorphism classification in a given tangential homotopy type is synonymous to finding inverses for the group $Im(\eta) \cap Im(\tau^*)$. That is what we try to achieve for $M=V \times S^k$. In general we find inverses for a big subgroup of that and in some favorable cases we find it for all.

\subsection{When $V = V_{2n+1,2}$}

We have the following diagram:
\begin{equation}\label{eq:5.1}\tag{5.1}
\begin{tikzcd}[column sep=large, row sep=1.8em]
{[S^{4n-1+k},\, G/O]}
  \arrow[d, "\pi_*"] 
&
S(V \times S^k)
  \arrow[d]
&
{[S^{4n-1},\, G/O]}
  \arrow[d, "q^*"]
\\
{[\Sigma^{k}V,\, G/O]}
  \arrow[r, "c^*"]
  \arrow[d, "l^*"]
&
{[V \times S^k,\, G/O]}
  \arrow[r, "j^*"]
&
{[V,\, G/O] \oplus [S^k,\, G/O]}
  \arrow[d, "k^*"]
\\
{[S^{2n-1+k} \cup _2e^{2n+k},\, G/O]}
&
&
{[S^{2n-1} \cup_2 e^{2n},\, G/O]}
\end{tikzcd}
\end{equation}
Let $\Omega_2\colon\pi_{4n-1+k}(S^{2n-1} \cup_2 e^{2n}) \xrightarrow[\mathrm{stab.}]{D \circ \Sigma^N} [S^{2n-1+k} \cup_2 e^{2n+k},SG] \xrightarrow{\tau^*} [S^{2n-1+k} \cup_2 e^{2n+k},G/O]$, and set $J = \mathrm{Im}(\Omega_2)$. We will find inverses for all elements in $([S^{4n-1+k},G/O] \oplus J \oplus [V,G/O] \oplus [S^k,G/O]) \cap \mathrm{Im}(\tau_*)$. If $\Omega_2\colon\pi_{4n-1+k}(S^{2n-1} \cup_2 e^{2n}) \xrightarrow{\Sigma^N} \pi_{4n-1+k}^s(S^{2n-1} \cup_2 e^{2n})$ is onto,  then we found inverses in $hS(V \times S^k)$ for all elements in $[V \times S^k, G/O]$ . We first show there is an inverse set $B_j$ of $j^*$ realised by normal invariants, then the same for $c^*([S^{4n-1+k},G/O] \oplus J)$, and finally that their sum is also realised by compositions.

First, we find the normal invariant of certain pinch maps (or at least their restrictions), using the method described in~\cite{crowley:kervaire}. Let $\alpha\colon S^{4n-1+k} \xrightarrow{a} S^{2n-1} \xrightarrow{i} S^{2n-1} \cup_2 e^{2n} \hookrightarrow V \hookrightarrow V \times S^k$ and $p_{\alpha}\colon V\times S^{k}\;\longrightarrow\;V\times S^{k}\vee S^{4n-1+k} \;\xrightarrow{\mathrm{Id}\vee \alpha}\; V\times S^{k}$. This leads to the following diagrams:
\begin{equation}\label{eq:5.2}\tag{5.2}
\begin{tikzcd}[column sep=2.3em, row sep=1.8em]
S^{4n-1+k}\times \mathbb{R}^N
  \arrow[r, "a \times \mathrm{id}"]
  \arrow[d, "\cong","\phi^1"{left}]
&
S^{2n-1}\times \mathbb{R}^N
  \arrow[r, " t \times \mathrm{id}"]
  \arrow[d, "\simeq","\cong","\phi^0"{left}]
&
V \times S^k \times \mathbb{R}^N
  \arrow[d, "\cong","\phi"{left}]
\\
\nu_{S^{4n-1+k}}
  \arrow[r, "b_a"']
  \arrow[d, "\cong","\psi^1"{left}]
&
l^* \nu_{V \times S^k}
  \arrow[r, "b_t"]
  \arrow[d,"\cong","\psi^0"{left}]
&
\nu _{V \times S^k}
  \arrow[d,  "\cong","\psi"{left}]
\\
S^{4n-1+k}
  \arrow[r, "a"']
&
S^{2n-1}
  \arrow[r, "t"']
&
V \times S^k
\end{tikzcd}
\end{equation}
\begin{equation}\label{eq:5.3}\tag{5.3}
\begin{tikzcd}[row sep=1.7em, column sep=2.3em]
{S^{4n-1+k+N} \,\vee\, S^N}
  \arrow[r,"\Sigma^N a \vee Id"]
  \arrow[d, "\simeq","q"{left}]
&
{S^{2n-1+N} \,\vee\, S^N}
  \arrow[d, "\simeq","p"{left}]
&
{} \\
{\Sigma^N S^{4n-1+k}_+}
  \arrow[r, "\Sigma^N a_+"]
  \arrow[d,"\cong","T(\psi^1)"{left}]
&
{\Sigma^N S^{2n-1}_+}
  \arrow[r,"\Sigma^N t_+"]
  \arrow[d, "\cong","T(\psi^0)"{left}]
&
{\Sigma^N (V \times S^k)}
  \arrow[d, "\cong","T(\psi)"{left}]
\\
{\mathrm{Th}(\nu_{4n-1+k})}
  \arrow[r, "\mathrm{Th}(b_a)"']
&
{\mathrm{Th}(l^*\nu_{V \times S^k})}
  \arrow[r, "\mathrm{Th}(b_t)"']
&
{\mathrm{Th}(\nu_{V \times S^k})}
\end{tikzcd}
\end{equation}

Call the left map $m$, middle map $h$. They are homotopy equivalences. After taking Spanier-Whitehead dual $D=D_{4n-1+k+2N}$, we see that, the corresponding element in $\{V \times S^k,S^0\}$ is : $g_a:\Sigma^{N}\!\left(V \times S^{k}\right) \xrightarrow{\;\Sigma^{N} t^!\;} \Sigma^{N}\operatorname{Th} (\nu(l)) \xrightarrow{\;D(T(b_\alpha))\;} \Sigma^{N} S^{2n-1}_+ \xrightarrow{c_{s_0}} S^{N}.$ The map is, $c^*(f_\alpha)$, for some $f_\alpha \in \{\Sigma^k V \vee S^k,S^0\}$. \cite[Lemma~7.3, 7.4]{crowley:kervaire}. Now the map on $S^k$ component is zero for factoring through higher dimensional wedge of spheres (5.4). So, $f_\alpha \in \{\Sigma^k V,S^0\}$. We want to see $l^*(f_\alpha)$,i.e. the restriction of $f_\alpha$ on $S^{2n-1+k} \cup _2e^{2n+k}$. We have the following diagram:
\begin{equation}\label{eq:5.4}\tag{5.4}
\begin{tikzcd}[column sep=large, row sep=large]
\Sigma^{N}\!\left(V \times S^{k}\right)
  \arrow[r, "\Sigma^{N} t^!"]
&
\Sigma^{N}\operatorname{Th}\bigl(\nu(l)\bigr)
  \arrow[r, "D(T(b_\alpha))"]
  \arrow[d, "\simeq","D(m)"{left}]
&
\Sigma^{N} S^{2n-1}_{+}
  \arrow[r, "c_{s_0}"]
  \arrow[d,"D(h)","\simeq"{left}]
&
S^{N}
\\
S^{2n-1+k+N} \cup_2 e^{2n+k+N}
  \arrow[u, "l"{left}]
  \arrow[r,"q"]
&
S^{2n+k+N} \vee S^{4n-1+k+N}
  \arrow[r, "\{a\}\,\vee\,\mathrm{id}"']
&
S^N \vee S^{4n-1+k+N}
\arrow[ur,bend right = 20]
\end{tikzcd}
\end{equation}
The map on bottom-left is $q$ as the map has to be non-zero by analyzing homology or cohomology and $q$ is only non-zero choice. So, $l^*(f_\alpha) = \pm q^*\{ a\}$.   
We have a commutative diagram:
\begin{equation}\label{eq:5.5}\tag{5.5}
\begin{tikzcd}
\cdots \arrow[r] &
\{S^{2n+k},\,S^{0}\} \arrow[r,"q^{*}"] \arrow[d,"Id= D"',"\cong"{right}] &
\{S^{2n-1+k}\cup e^{2n+k},\,S^{0}\} \arrow[r,"i^{*}"] \arrow[d,"\cong\{D\}"] &
\{S^{2n-1+k},\,S^{0}\} \arrow[r] \arrow[d,"Id= D"',"\cong"{right}] &
\cdots
\\
\cdots \arrow[r] &
\{S^{4n-1+k},\,S^{2n-1}\} \arrow[r,"i^{*}"] &
\{S^{4n-1+k}, S^{2n-1}\cup_2 e^{2n}\} \arrow[r,"q^{*}"] &
\{S^{4n-1+k},\,S^{2n}\} \arrow[r] &
\cdots 
\end{tikzcd}
\end{equation}
Let $\{b\} \in \mathrm{Im}(q^*)$ with $q^*(\{\hat{b}\}) = \{b\}$. Take the map $f_{b}\colon S^{4n-1+k} \xrightarrow{\hat{b}} S^{2n-1}\cup_2 e^{2n} \hookrightarrow V \hookrightarrow V \times S^k$. We want to analyse the normal invariant of this map. We again have the following diagrams (the intermediate submanifold here is $V$, not $S^{2n-1}\cup_2 e^{2n}$):
\begin{equation}\label{eq:5.6}\tag{5.6}
\begin{tikzcd}[column sep=2.3em, row sep=1.8em]
S^{4n-1+k}\times \mathbb{R}^N
  \arrow[r, "\hat{b} \times \mathrm{id}"]
  \arrow[d, "\cong","\phi^1"{left}]
&
V\times \mathbb{R}^N
  \arrow[r, " t \times \mathrm{id}"]
  \arrow[d, "\simeq","\cong","\phi^0"{left}]
&
V \times S^k \times \mathbb{R}^N
  \arrow[d, "\cong","\phi"{left}]
\\
\nu_{S^{4n-1+k}}
  \arrow[r, "b_{\hat{b}}"']
  \arrow[d, "\cong","\psi^1"{left}]
&
t^* \nu_{V \times S^k}
  \arrow[r, "b_t"]
  \arrow[d,"\cong","\psi^0"{left}]
&
\nu _{V \times S^k}
  \arrow[d,  "\cong","\psi"{left}]
\\
S^{4n-1+k}
  \arrow[r, "\hat{b}"']
&
V
  \arrow[r, "t"']
&
V \times S^k
\end{tikzcd}
\end{equation}
\begin{equation}\label{eq:5.7}\tag{5.7}
\begin{tikzcd}[row sep=1.5em, column sep=2em]
&
S^{4n-1+k+N} \vee S^N 
    \arrow[r, "\Sigma^N \hat{b} \vee \mathrm{Id}"] 
    \arrow[d, "\simeq"] 
& 
{\begin{matrix} 
        S^{2n-1+N} \cup_2 e^{2n+N} \\ 
        \vee \; S^{4n-1+N} \vee S^N 
    \end{matrix}} 
    \arrow[d, "\simeq"] \\
&
\Sigma^N S^{4n-1+k}_+ 
    \arrow[d, "\simeq"]
    \arrow[r,"\Sigma^N \hat{b}_+"]
& 
\Sigma^N V_+ 
    \arrow[d, "\simeq"] \\
& \mathrm{Th}(S^{4n-1+k} \times \mathbb{R}^N) 
    \arrow[r] 
    \arrow[d, "\simeq"] 
& \mathrm{Th}(V \times \mathbb{R}^N) 
    \arrow[d] \\
S^{4n-1+k+N} 
   \arrow[uuur, bend left=50] 
   \arrow[uur, bend left=20] 
   \arrow[ur] 
   \arrow[r]
& \mathrm{Th}(\nu_{S^{4n-1+k}})  \arrow[r,"T(b_{\hat{b}})"] 
& \mathrm{Th}(t^* \nu_{V \times S^k})
\end{tikzcd} 
\end{equation}

After taking Spanier-Whitehead dual $D=D_{4n-1+k+2N}$, we see that, the corresponding element in $\{V \times S^k,S^0\}$ is : $g_{\hat{b}}:\Sigma^{N}\!\left(V \times S^{k}\right) \xrightarrow{\;\Sigma^{N} t^!\;} \Sigma^{N}\operatorname{Th} (\nu(t)) \xrightarrow{\;D(T(b_{\hat{b}}))\;} \Sigma^{N} S^{4n-1+k}_+ \xrightarrow{c_{s_0}} S^{N}.$ {\cite[Lemma~7.3, 7.4]{crowley:kervaire}}. The map is, $c^*(f_{\hat{b}})$, for some $f_{\hat{b}} \in \{\Sigma^k V \vee S^k,S^0\}$. But, from the diagram below the map is zero on $S^k$. so, $f_{\hat{b}} \in \{\Sigma^k V , S^0\}$ We want to see $i^*l^*(f_{\hat{b}})$,i.e. the restriction of $f_{\hat{b}}$ on $S^{2n-1+k}$. We have the following diagram:
\begin{equation}\label{eq:5.8}\tag{5.8}
\begin{tikzcd}[row sep=2em, column sep=large]
\Sigma^N (V \times S^k) 
    \arrow[r, "t!"] 
    \arrow[d, "\simeq"'] 
& \Sigma^N \mathrm{Th}(\nu_t) 
    \arrow[r] 
    \arrow[d, "\simeq"] 
& \Sigma^N S^{4n-1+k}_+ 
    \arrow[d, "\simeq"] \arrow[r] 
& S^N    
\\
\begin{aligned}
    & S^{2n-1+k+N} \cup_2 e^{2n+k+N} \\
    & \vee S^{k+N} \vee S^{4n-1+k+N} \\
    & \vee (S^{2n-1+N} \cup_2 e^{2n+N})  \vee S^{4n-1+N}
\end{aligned}
 \arrow[r, "p","\simeq"{below}] 
& \begin{matrix}
    S^{2n-1+k+N} \cup_2 e^{2n+k+N}\\ \vee S^{k+N} \\ \vee S^{4n-1+k+N}
\end{matrix} 
    \arrow[r, "D(\hat{b}) \vee \mathrm{Id}",{below}]
    \arrow[r, dashed, bend left=20, "D(\Sigma^N \hat{b})"] 
    \arrow[r, dashed, bend right=55, "\mathrm{Id}"'] 
& S^N \vee S^{4n-1+k+N} \arrow[ur,bend right = 25]
\end{tikzcd}
\end{equation}
Notice the following:

\begin{enumerate}[label=\textup{\arabic*)}]
\item The map is $c^*(f_{\hat{b}})$ for some $f_{\hat{b}}$, where $c\colon V \times S^k \to \Sigma^k V$ is the quotient map; we need only know it on $S^{2n-1+k+N} \cup_2 e^{2n+k+N} \vee S^{4n-1+k+N}$ (zero elsewhere), and we further restrict to $S^{2n-1+k}$.
\item On $S^{2n-1+k+N} \cup_2 e^{2n+k+N}$, the map $p$ is $f \vee g \colon S^{2n-1+k+N} \cup_2 e^{2n+k+N} \to S^{2n-1+k+N} \cup_2 e^{2n+k+N} \vee S^{k+N}$, where $f$ must be a homotopy equivalence; $g$ does not contribute since the next map is zero on $S^{k+N}$.
\item The final map is $S^{2n-1+k+N} \hookrightarrow S^{2n-1+k+N} \cup_2 e^{2n+k+N} \xrightarrow{\pm\mathrm{Id}} S^{2n-1+k+N} \cup_2 e^{2n+k+N} \xrightarrow{\pm D\{\hat{b}\}} S^N$. Whatever the choice, this equals $\pm i^*D\{\hat{b}\} = D(b)$.
\end{enumerate}

\begin{theorem}[See Diagram~\ref{eq:5.5}]\label{thm:5.1}
\leavevmode\begin{enumerate}[label=\textup{\arabic*)}]
\item If $\{\alpha\} = i^*(\{a\})$ where $a \in [S^{4n-1+k}, S^{2n-1}]$, then $\tilde{\eta}^t(P_{\pm i_*a},\, b_t \circ b_a) = c^*(f_a)$ with $l^*(f_a) = q^*\{a\}$, and $\eta^t(p_{\pm i_*a},\, \mathrm{Id}\,\#\,b_t \circ b_a) = c^*(H(f_\alpha))$ with $l^*(H(f_\alpha)) = q^*(H(\{a\}))$. The first lies in the $\{\mathord{-},S^0\}$ sequence and the second in $[\mathord{-},SG]$.
\item Let $\{b\} \in \mathrm{Im}(q^*)$ with $q^*(\{\hat{b}\}) = \{b\}$. Then $\tilde{\eta}^t(P_{\pm i_*\hat{b}},\, b_{\hat{b}} \circ b_t) = c^*(f_b)$ with $i^*l^*(f_b) = D(\{b\})$, and $\eta^t(p_{\pm i_*\hat{b}},\, \mathrm{Id}\,\#\,b_{\hat{b}} \circ b_t) = c^*(H(f_b))$ with $i^*l^*(H(f_b)) = H(D(\{b\}))$. The first lies in $\{\mathord{-},S^0\}$ and the second in $[\mathord{-},SG]$.
\item We can realise $([S^{4n-1+k},G/O] \oplus J \oplus [V,G/O] \oplus [S^k,G/O]) \cap \mathrm{Im}(\tau_*)$ by normal invariants, given that $\Omega_1$ is surjective onto $\mathrm{Im}(\tau_*)$.
\end{enumerate}
\end{theorem}
\begin{equation}\label{eq:5.9}\tag{5.9}
\begin{tikzcd}[row sep=1.4em, column sep=3em]
&{[\Sigma^{k}V,\, G/O]}
  \arrow[r,"c^{*}"]
  \arrow[d,"l^{*}"']
&
{[V\times S^{k},\, G/O]}\\
{[S^{2n+k},\, G/O]}
  \arrow[r,"q^{*}"]
&
{[S^{2n-1+k}\cup_{e} e^{2n+k},\, G/O]}
  \arrow[r,"i^{*}"]
&
{[S^{2n-1+k},\, G/O]}
\\
{[S^{2n+k},\, SG]}
  \arrow[r,"q^{*}"]
  \arrow[u,"\tau^*"]
&
{[S^{2n-1+k}\cup_{e} e^{2n+k},\, SG]}
  \arrow[r,"i^{*}"]
  \arrow[u,"\tau^*"]
&
{[S^{2n-1+k},\, SG]}
  \arrow[u,"\tau^*"]
\\
\{S^{2n+k},\, S^{0}\}
  \arrow[r,"q^{*}"]
  \arrow[u,"H","\equiv"{right}]
&
\{S^{2n-1+k}\cup_{e} e^{2n+k},\, S^{0}\}
  \arrow[r,"i^{*}"]
  \arrow[u,"H","\equiv"{right}]
&
\{S^{2n-1+k},\, S^{0}\}
  \arrow[u,"H","\equiv"{right}]\\
\{S^{4n-1+k},\, S^{2n-1}\}
  \arrow[r, "i^{*}"]
  \arrow[u,"D=Id","\cong"{right}]
&
\{S^{4n-1+k},\, S^{2n-1}\cup_{e} e^{2n}\}
\arrow[r, "q^{*}"]
\arrow[u,"D","\cong"{right}]
&
\{S^{4n-1+k},\, S^{2n}\}
\arrow[u,"D=Id","\cong"{right}]
\\
{[S^{4n-1+k},\, S^{2n-1}]}
  \arrow[r,"i^*",dashed]
  \arrow[u, "\Sigma_1^{*}"]
&
{[S^{4n-1+k},\, S^{2n-1}\cup_{e} e^{2n}]}
  \arrow[r,"q^*",dashed]
  \arrow[u, "\Sigma_2^{*}"]
&
{[S^{4n-1+k},\, S^{2n}]}
  \arrow[u, "\Sigma_3^{*}"]
\end{tikzcd}
\end{equation}
The left vertical map is called $\Omega_1$, the next $\Omega_2$, and the third $\Omega_3$. We set $\tau^*HD = R$.

\begin{proof}
Only statement~3) requires proof.

\noindent\textbf{A.} We first find an inverse set $E'$ of $\mathrm{Im}(\Omega_2)$ for the map $l^*$, with $E' \subseteq [\Sigma^k V,G/O]$, and show that $c^*(E')$ is fully realised by normal invariants of pinch maps.

Let $y \in [S^{4n-1+k}, S^{2n-1} \cup_2 e^{2n}]$ and let $R\{y\}$ be a generic element of $\operatorname{Im}\Omega_2$. Set $q_* y = b_y$. For any element in $\operatorname{Im}(q_* \circ \Sigma_2^*) = \operatorname{Im}(\Sigma_3^* q_*)$ we can choose an inverse of the form $\{x\} \in \{S^{4n-1+k}, S^{2n-1} \cup_2 e^{2n}\}$; for $\{c\} \in \operatorname{Im}(\Sigma_3^* q_*)$ call this inverse $\{\hat{c}\}$. In particular, $\{b_y\}$ has an inverse $\{\hat{b}_y\}$.
First, from the normal invariants of the pinch maps, we have:
\[
\eta^t(P_{\pm i_* \hat{b}_y}, \text{Id} \mathbin{\#} b_{\hat{b}_y} \circ b_t) = c^*(H(f_{\hat{b}_y})) \quad \text{and} \quad i^* l^*(H(f_{\hat{b}_y})) = H(D\{ b_y\}).
\]
Passing this to $G/O$, we get:
\[
\eta(P_{\pm i_* \hat{b}_y}) = c^*(\tau_* H(f_{l_y})),
\]
and checking the restriction yields:
\[
i^* l^*(\tau_* H(f_{b_y})) = \tau_* i^* l^*(H(f_{b_y})) = \tau_* H(D\{ b_y\}) = R(\{b_y\}).
\]
We also know $i^* R(\{y\}) = R(\{b_y\})$. Applying the long exact sequence, the difference yields:
\[
l^* H(f_{b_y}) - H D(\{y\}) = q^* \alpha \quad \text{for some} \quad \alpha \in [S^{2n+k}, SG].
\]
Passing this to $G/O$ via $\tau_*$ gives:
\begin{align*}
l^* \tau_* \bigl(H(f(b_y))\bigr) - R(\{y\}) &= \tau_* l^* \bigl(H(f(b_y))\bigr) - \tau_* H D(\{y\}) \\
&= \tau_* q^* \alpha \\
&= q^* \tau_* \alpha.
\end{align*}
Rearranging this gives $R(\{y\}) = l^*(\tau_* H(f(b_y))) + q^* \tau_* \alpha$.

Because $\Omega_1$ is surjective onto $\operatorname{Im} \tau_*$, we can set $q^* \tau_* \alpha = q^* \tau_* (H(D\{a_y\}))$, giving:
\[
R(\{y\}) = l^*(\tau_* H(f(b_y))) + q^* \tau_* \bigl(H(\{a_y\})\bigr).
\]

Similarly, for the other component we have:
\[
\eta^t(P_{\pm i_* a_y}, \text{Id} \mathbin{\#} b_l \circ b_{a_y}) = c^*(H(f_{a_y})) \quad \text{and} \quad l^*(H(f_{a_y})) = q^*(H\{a_y\}).
\]

Passing to $G/O$, we obtain $\eta(P_{\pm i_* a_y}) = \tau_* c^* (H(f_{a_y})) = c^* (\tau_* H(f_{a_y}))$, alongside:
\[
l^* (\tau_* H(f_{a_y})) = \tau_* (l^* H(f_{a_y})) = \tau_* q^* (H(\{a_y\})) = q^* \tau_* (H(\{a_y\})).
\]
Substituting this back completes the isolation of $R(\{y\})$:
\begin{align*}
R(\{y\}) &= l^* (\tau_* H(f(b_y))) + l^* (\tau_* H(f_{a_y})) \\
&= l^* \bigl(\tau_* H(f(b_y)) + \tau_* H(f_{a_y})\bigr).
\end{align*}
Thus $E' = \{ \tau_* H(f(b_y)) + \tau_* H(f_{a_y}) \mid R(\{y\}) \in \operatorname{Im} \Omega_2 \}$ is the inverse set of $\operatorname{Im}(\Omega_2)$.
We now show that $c^*(E')$ is realised by normal invariants.
$\eta(P_{\pm i_* a_y}) = c^* (\tau_* H(f_{a_y}));$ $\eta(P_{\pm i_* \hat{b}_y}) = c^* (\tau_* H(f(b_y)))$.
Indeed, $\eta(P_{\pm i_* \hat{b}_y} \circ P_{\pm i_* a_y}) = c^*(\tau_* H(f(b_y))) + c^*(\tau_* H(f_{a_y})) = c^*(x_{E'})$ for some $x_{E'} \in E'$. The splitting into a sum for the normal invariant comes from the analyzing the following map:
\begin{align}
S^{2n-1+k+N} \cup_2 e^{2n+k+N} & \qquad \qquad S^{2n-1+k+N} \cup_2 e^{2n+k+N} \nonumber\\
\vee S^{k+N} \vee S^{4n-1+k+N} & \xrightarrow{P_{\mp i_* \hat{b}_y}} \vee S^{k+N} \vee S^{4n-1+k+N} \xrightarrow{\eta^t(P_{\pm i_* a_y})} SG \to G/O\nonumber\\
\vee S^{2n-1+N} \cup_2 e^{2n+N} & \qquad \qquad \vee S^{2n-1+N} \cup_2 e^{2n+N} \nonumber\\
\vee S^{4n-1+N} & \qquad \qquad \vee S^{4n-1+N} \tag{5.10}\label{eq:5.10}
\end{align}
\\
$\eta(P_{\pm i_* {a}_y})$ non-zero on $S^{2n-1+k+N} \cup_2 e^{2n+k+N} \vee S^{4n-1+k+N}$
And, $\hat{b}_y : S^{4n-1+k+N} \longrightarrow S^{2n-1+N} \cup_2 e^{2n+N}$. So,
The composition is $\eta(P_{\pm i_* {a}_y})$.
Hence $c^*(E')$ is fully realised by normal invariants of pinch maps.

\begin{remark}\label{rem:A}
The same conclusion holds if $\Pi \circ \Omega_1$ is surjective onto $\operatorname{Im}(\Pi \circ \tau_*)$, where $\Pi$ is the quotient from the long exact sequence. Thus surjectivity on two-torsion elements suffices.
\end{remark}
\begin{proof}
$R(\{y\}) - l^* \tau_* H(f_{b_y}) = q^* \tau_* \alpha = \bar{q}_* \pi(\tau_* \alpha) = q_* \tau_* H(\{a_y\})$ for some $a_y \in [S^{4n-1+k}, S^{2n-1}]$, giving $R(\{y\}) = l^* \tau_* H(f_{b_y}) + q_* \tau_* H(\{a_y\})$; the rest of the argument is identical.
\end{proof}

\noindent\textbf{B.} For any $(\alpha, \beta) \in ([V, G/O] \oplus [S^k, G/O]) \cap \mathrm{Im}\,\tau_*$, the element $\pi_V^*\alpha + \pi_S^*\beta$ is realised by normal invariants.
\begin{equation}\label{eq:5.11}\tag{5.11}
\begin{tikzcd}[row sep=large, column sep=large]
  {[S^{4n-1+k}, G/O]} \arrow[r, "q^*"]
  & {[\Sigma^k V, G/O]} \arrow[r, "c^*"]
  & {[V \times S^k, G/O]} \arrow[r]
  & {[V, G/O] \oplus [S^k, G/O]} \arrow[l, bend left, dashed, "\pi_S^*"'] \arrow[l, bend right, dashed, "\pi_V^*"']
\end{tikzcd}
\end{equation}

Let $\eta(\phi) = \alpha$, $\eta(c_{\Sigma_k}) = \beta$ (we can do that because $k \neq 2^i-2$), where $c_{\Sigma_k}\colon \Sigma^k \xrightarrow{\cong} S^k$.
Set $\varphi = \varphi^{(\Sigma, g)}\colon V \# \Sigma \xrightarrow{h_{\Sigma}} V \xrightarrow{P(g)} V$ and
$G^{(\Sigma, g, \Sigma^k)}\colon V \# \Sigma \times \Sigma^k \xrightarrow{\mathrm{Id}\times c_{\Sigma^k}} V \# \Sigma \times S^k \xrightarrow{\varphi^{(\Sigma, g)} \times \mathrm{Id}} V \times S^k.$
Then $\eta(G^{\Sigma, g, \Sigma^k}) = \eta(\varphi \times \mathrm{Id}) + ((\varphi \times \mathrm{Id})^{-1})^* \eta(\mathrm{Id} \times c_{\Sigma^k}).$
Now, $\eta(\varphi \times \mathrm{Id}) = \pi_V^*(\eta(\varphi^{(\Sigma, g)}))$, where $\pi_V\colon V \times S^k \to V$ is the projection (the normal invariant of a product $f \times \mathrm{Id}_B$ equals $\pi_1^*(\eta(f))$; see~\cite[\S 3.3]{luck}). Also,
$\eta(\mathrm{Id} \times c_{\Sigma^k}) = \tilde{\pi}_S^*(\eta(c_{\Sigma^k}))$, where $\tilde{\pi}_S\colon V \# \Sigma \times S^k \to S^k$.
$((\varphi \times \mathrm{Id})^{-1})^* \tilde{\pi}_S^*(\eta(c_{\Sigma^k})) = \pi_S^*(\eta(c_{\Sigma^k}))$ from the commutative diagram below, which gives
$\eta(G^{(\Sigma, g, \Sigma^k)}) = \pi_V^*(\eta(\varphi^{(\Sigma, g)})) + \pi_S^*(\eta(c_{\Sigma^k})).$
\begin{equation}\label{eq:5.12}\tag{5.12}
\begin{tikzcd}
V \# \Sigma \times S^k \arrow[r, "\varphi \times Id"] \arrow[d, "\tilde{\pi}_S"'] & V \times S^k \arrow[d, "\pi_S"] \\
S^k \arrow[r, equal, "Id"'] & S^k
\end{tikzcd}
\quad \Rightarrow \quad
\begin{tikzcd}
\left[ V \times S^k, \frac{G}{O} \right] \arrow[r, "(\varphi \times Id)^*"] & \left[ V \# \Sigma \times S^k, \frac{G}{O} \right] \\
\left[ S^k, \frac{G}{O} \right] \arrow[u, "\pi_S^*"] \arrow[r, equal, "Id"'] & \left[ S^k, \frac{G}{O} \right] \arrow[u, "\tilde{\pi}_S^*"']
\end{tikzcd}
\end{equation}
\noindent\textbf{C.} We show that any element of $[c^*(E') + \pi_V^{*}([V,G/O]) + \pi_S^{*}([S^k,G/O])] \cap \mathrm{Im}(\tau^*)$ is realised by normal invariants of tangential homotopy equivalences.

Let $x = c^*(x_{E'}) + (\pi_V^*(\alpha) + \pi_S^*(\beta))$ be a generic element in this subset, where $x_{E'} \in E'$. We wish to generate this element via normal invariants.

Consider the following composition of maps:
\[
V \mathbin{\#} \Sigma \times \Sigma^k \xrightarrow{h} V \times S^k \xrightarrow[P_g]{m} V \times S^k
\]
where $h = (\mathrm{Id} \times c_{\Sigma^k}) \circ (\varphi^{(\Sigma, g)} \times \mathrm{Id})$ and $g$ is the composition $S^{4n-1+k} \xrightarrow{\tilde{g}} S^{2n-1} \cup_2 e^{2n} \hookrightarrow V \hookrightarrow V \times S^k$. 
Here, we have $\eta(P_g) = c^*(x_{E'})$. 

Applying the composition formula yields:
\[
\eta(m \circ h) = \eta(m) + (m^{-1})^* \eta(h).
\]
Now, we need to precisely identify the second term. It corresponds to the map:
\[
(m^{-1})^* \eta(h) : V \times S^k \xrightarrow[P(-g)]{m^{-1}} V \times S^k \xrightarrow{\eta^t(h)} SG \xrightarrow{\tau} G/O.
\]
Because of the structure of the pinch map $P(-g)$, the resulting map takes the form:
\[
\eta(h) - c^* (\eta(h) \circ \tilde{g}) = \eta(h) - c^* g^* (\eta(h)).
\]

Substituting this back gives:
\[
\eta(m \circ h) = \eta(m) + \eta(h) - c^* g^* (\eta(h)).
\]
Observe that $g^*(\eta(h)) \in \operatorname{Im} (\tau_*) \cap [S^{4n-1+k}, G/O]$. Consequently, $g^* \eta(h) = \eta_S([\Sigma_x^{4n-1+k}])$ for some normal invariant of a sphere. We can then modify the composition iteratively:
\begin{align*}
\eta(m \circ h \mathbin{\#} h_{\Sigma_x^{4n-1+k}}) &= \eta(m) + \eta(h) \\
&= c^*(x_{E'}) + (\pi_V^*(\alpha) + \pi_S^*(\beta)).
\end{align*}

This explicitly demonstrates that the element is realized by the normal invariants of the following full tangential homotopy equivalence:
\[
m \circ h \ \# \ h_{\Sigma_{x_{0}}} : (V \# \Sigma \times \Sigma^k) \# \Sigma_{x'}^{4n-1+k} \# \Sigma_x^{4n-1+k} \xrightarrow{h_{(\Sigma_{x'} \# \Sigma_{x})}} V \# \Sigma \times \Sigma^k \xrightarrow[(Id \times c_{\Sigma^k}) \circ (\varphi^{(\Sigma, g)} \times Id)]{h} V \times S^k \xrightarrow[P(g)]{m} V \times S^k
\]
Normal invariant of these composite maps fully realizes $c^*(E') + \bigl([V, G/O] \oplus [S^k, G/O]\bigr) \cap \operatorname{Im}(\tau^*)$. This proves part C.

\noindent\textbf{D.} We prove that $[c^*((i^*)^{-1}(J)) + \pi_V^{*}([V,G/O]) + \pi_S^{*}([S^k,G/O])] \cap \mathrm{Im}(\tau^*)$ is realised by normal invariants of tangential homotopy equivalences.
\begin{equation}\label{eq:5.13}\tag{5.13}
\begin{tikzcd}
\left[ \Sigma^k V, G/O \right] \arrow[r, "c^*"] \arrow[d, "i^*"'] & \left[ V \times S^k, G/O \right] \arrow[r] & \left[ V, G/O \right] \oplus \left[ S^k, G/O \right] \arrow[l, bend right, dashed, "\pi_V^*"'] \arrow[l, bend left, dashed, "\pi_S^*"] \\
\left[ S^{2n-1+k} \cup_2 e^{2n+k}, G/O \right] & &
\end{tikzcd}
\end{equation}
Let $x \in [c^*((i^*)^{-1}(J)) + \pi_V^{*}([V,G/O]) + \pi_S^{*}([S^k,G/O])] \cap \mathrm{Im}(\tau^*) \subset [V \times S^k, G/O] \cap \operatorname{Im} \tau_*$, where, $J = \operatorname{Im} \Omega_2$. 
Then $x$ decomposes as $x = c^*(x_J) + (\pi_V^*(\alpha) + \pi_S^*(\beta))$, such that $i^* x_J \in J$. 

Let $i^*(x_J) = i^*(x_{E'})$ for some $x_{E'} \in E'$, since $E'$ was chosen as a set of inverses for $J = \operatorname{Im}(\Omega_2)$. 
This implies that:
\[ x_J - x_{E'} = (q)^*(A_{x_J}) \]
Observe the following logical chain:
\begin{itemize}
    \item $x_{E'} \in E' \subset \operatorname{Im} \tau_*$ and $x_J \in \operatorname{Im} \tau_*$ implies $x_J - x_{E'} \in \operatorname{Im} \tau_*$.
    \item Hence, $x_J - x_{E'} \in \left[ \Sigma^k V, G/O \right] \cap \operatorname{Im} \tau_*$.
    \item Because $x_J - x_{E'} = q^*(A_{x_J})$, we deduce that $A_{x_J} \in \operatorname{Im} \tau_*$.
\end{itemize}
(This follows because all relevant short exact sequences split naturally; an element is in $\operatorname{Im}(\tau_*)$ if and only if its individual coordinates are).

Returning to $x = c^*(x_J) + (\pi_V^*(\alpha) + \pi_S^*(\beta))$, let us define  a new element:
$ x' = c^*(x_{E'}) + (\pi_V^*(\alpha) + \pi_S^*(\beta)) $.
Suppose $\eta(m' \circ h' \mathbin{\#} h_{\Sigma_{x'}}) = x'$. It is possible because of C. Then the difference $x - x'$ is given by:
\[ x - x' = c^*(x_J - x_{E'}) = c^*q^*(A_{x_J}) \]
Since $A_{x_J} \in \operatorname{Im}\bigl(\tau_*\colon [S^{4n-1+k}, SG] \to [S^{4n-1+k}, G/O]\bigr)$, there exists a homotopy sphere such that $\eta([\Sigma_x])=A_{x_{_J}}$(again it is possible because dimension of the $V \times S^k \neq 2^i-2$). We then compute:
\begin{align*}
    \eta(m' \circ h' \mathbin{\#} h_{\Sigma_{x'}} \mathbin{\#} h_{\Sigma_x}) 
    &= \eta(m' \circ h' \mathbin{\#} h_{\Sigma_{x'}}) + \eta(h_{\Sigma_x}) \\
    &= x' + (c \circ q)^* A_{x_J} \\
    &= x
\end{align*}

Thus, the element $x$ is explicitly realized by the normal invariants of the following composed tangential homotopy equivalences:
\[
(V \mathbin{\#} \Sigma \times \Sigma^k) \mathbin{\#} \Sigma_{x'}^{4n-1+k} \mathbin{\#} \Sigma_x^{4n-1+k} \xrightarrow{h_{(\Sigma_{x'} \mathbin{\#} \Sigma_{x})}} V \mathbin{\#} \Sigma \times \Sigma^k \xrightarrow{h} V \times S^k \xrightarrow[P_g]{m} V \times S^k
\]
where $h = (\mathrm{Id} \times c_{\Sigma^k}) \circ (\varphi^{(\Sigma, g)} \times \mathrm{Id})$, and the final composition simplifies to $m \circ h \mathbin{\#} h_{\Sigma_x}$. This completes the proof of statement~3).
\end{proof}

\begin{corollary}\label{cor:5.3}
If $\Omega_1$ and $\Omega_2$ are surjective onto $\mathrm{Im}(\tau_*)$ , then every smooth manifold tangentially homotopic to $V \times S^k$ is almost diffeomorphic to $V \# \Sigma \times \Sigma^k$.
\end{corollary}

\begin{remark}\label{rem:5.4}
There is limited information on homotopy groups of Moore spaces. While the first condition is often satisfied, it is unclear when $\Omega_2$ is surjective onto $Im(\tau_*)$. Nevertheless, the theorem remains useful for finding all inverses and completing the classification. One possible approach is to find representatives of generators of $([S^{2n-1+k}\cup_{e} e^{2n+k},\, G/O] \cap \mathrm{Im}(\tau^*))/ \mathrm{Im}(\Omega_2)$ inside $Im (\eta)$, find their $\eta$-inverses, and show that all of $[S^{2n-1+k}\cup_{e} e^{2n+k},\, G/O] \cap \mathrm{Im}(\tau^*)$ is realised by normal invariants of suitable compbinations.
\end{remark}

\subsection{When $V = V_{2n+2,2}$}

When $V = V_{2n+2,2} \simeq (S^{2n} \vee S^{2n+1}) \cup_\rho e^{4n+1}$, we have analogous results for $V \times S^k$. Let $\Psi_1 \colon \pi_{4n+1+k}(S^{2n}) \to [S^{2n+1+k},SG] \to [S^{2n+1+k},G/O]$ and $\Psi_2 \colon \pi_{4n+1+k}(S^{2n+1}) \to [S^{2n+k},SG] \to [S^{2n+k},G/O]$. Set $J_1 = \mathrm{Im}(\Psi_1)$ and $J_2 = \mathrm{Im}(\Psi_2)$. By the same methods as above:

\begin{theorem}\label{thm:5.5}
We can realise $([S^{4n+1+k},G/O] \oplus J_1 \oplus J_2 \oplus [V,G/O] \oplus [S^k, G/O]) \cap \mathrm{Im}(\tau_*)$ by normal invariants of tangential homotopy equivalences of the form
\begin{equation}\tag{5.14}\label{eq:5.14}
(V \# \Sigma \times \Sigma^k) \# \Sigma_{x'}^{4n+1+k} \# \Sigma_x^{4n+1+k} \xrightarrow{h_{(\Sigma_{x'} \# \Sigma_{x})}} V \# \Sigma \times \Sigma^k \xrightarrow[(\mathrm{Id} \times c_{\Sigma^k}) \circ (\varphi^{(\Sigma,g)} \times \mathrm{Id})]{h} V \times S^k \xrightarrow[P_g]{m} V \times S^k.
\end{equation}
\end{theorem}
\begin{corollary}\label{cor:5.6}
If $\Psi_1$ and $\Psi_2$ are surjective onto $Im(\tau_*)$, then every smooth manifold tangentially homotopic to $V \times S^k$ is almost diffeomorphic to $V \# \Sigma \times \Sigma^k$. In particular, any smooth manifold tangentially homotopy equivalent to $V_{12,2} \times S^3$, $V_{16,2} \times S^3$, or $V_{10,2} \times S^5$ looks like $V \# \Sigma \times S^k$.
\end{corollary}

\begin{remark}\label{rem:5.7}
In many cases the conditions are satisfied ~\cite{toda} ~: for example, $V_{12,2} \times S^3$, $V_{16,2} \times S^3$, $V_{10,2} \times S^5$. By Corollary~\ref{cor:5.6}, all smooth manifolds tangentially homotopic to $V \times S^k$ are then homeomorphic to $V \# \Sigma \times S^k$. Note that $[S^3,G/O] = [S^5,G/O] = 0$~\cite[Remark~9.22]{ranicki}. There may be many further cases in high dimensions where this applies, but a complete check is limited by our knowledge of homotopy groups of spheres. Even without surjectivity, the theorem is still useful for finding all inverses and completing the classification.
\end{remark}

\section*{Acknowledgements}
The author is deeply grateful to his doctoral advisor Dr.\ Ramesh Kasilingam (IIT Madras) for many valuable discussions throughout this work.
The first part of this research was supported by CSIR-NET fellowship.


\begin{thebibliography}{99}

\bibitem{adams:J}
J.~F.~Adams,
\emph{On the groups $J(X)$, I--IV},
Topology \textbf{2} (1963), 181--195; \textbf{3} (1965), 137--171; \textbf{3} (1965), 193--222; \textbf{5} (1966), 21--71.

\bibitem{browder}
W.~Browder,
\emph{Surgery on Simply-Connected Manifolds},
Ergebnisse der Mathematik und ihrer Grenzgebiete, Band~65,
Springer-Verlag, New York--Heidelberg, 1972.

\bibitem{crowley:kervaire}
D.~Crowley and I.~Hambleton,
\emph{Finite group actions on Kervaire manifolds},
Adv.\ Math.\ \textbf{283} (2015), 88--129. \texttt{arXiv:1305.6546}.

\bibitem{crowley:7mflds}
D.~Crowley,
\emph{The classification of highly connected manifolds in dimensions $7$ and $15$},
Ph.D.\ thesis, Indiana University, 2002. \texttt{arXiv:math/0203253}.

\bibitem{desapio}
R.~De~Sapio,
\emph{Almost diffeomorphisms of manifolds},
Pacific J.\ Math.\ \textbf{26} (1968), 47--56.

\bibitem{funcspaces}
E.~H.~Spanier,
\emph{Function spaces and duality},
Ann.\ of Math.\ (2) \textbf{70} (1959), 338--378.

\bibitem{james}
I.~M.~James,
\emph{The Topology of Stiefel Manifolds},
London Math.\ Soc.\ Lecture Note Ser., vol.~24,
Cambridge University Press, Cambridge, 1976.

\bibitem{luck}
W.~L\"uck,
\emph{A basic introduction to surgery theory},
ICTP Lecture Notes, Abdus Salam ICTP, Trieste, 2002.
Available at \texttt{https://him-lueck.uni-bonn.de/data/ictp.pdf}.

\bibitem{madsen}
I.~Madsen, L.~R.~Taylor, and B.~Williams,
\emph{Tangential homotopy equivalences},
Comment.\ Math.\ Helv.\ \textbf{55} (1980), 445--484.

\bibitem{may}
J.~P.~May,
\emph{A Concise Course in Algebraic Topology},
Chicago Lectures in Mathematics,
University of Chicago Press, Chicago, 1999.

\bibitem{nomura}
Y.~Nomura,
\emph{Self homotopy equivalences of Stiefel manifolds $W_{m,2}$ and $V_{m,2}$},
Osaka J.\ Math.\ \textbf{20} (1983), 79--93.

\bibitem{ottenberger}
S.~Ottenburger,
\emph{Simply and tangentially homotopy equivalent but non-homeomorphic homogeneous manifolds},
Asian J.\ Math.\ \textbf{27} (2023), no.~1, 57--76. \texttt{arXiv:1102.5708}.

\bibitem{ravenel}
D.~C.~Ravenel,
\emph{Complex Cobordism and Stable Homotopy Groups of Spheres},
Pure and Applied Mathematics, vol.~121,
Academic Press, Orlando, 1986.

\bibitem{ranicki}
A.~Ranicki,
\emph{Algebraic and Geometric Surgery},
Oxford Mathematical Monographs,
Oxford University Press, Oxford, 2002.

\bibitem{toda}
H.~Toda,
\emph{Composition Methods in Homotopy Groups of Spheres},
Ann.\ of Math.\ Studies, vol.~49,
Princeton University Press, Princeton, 1962.

\bibitem{homotopy5mfld}
P.~Li and Z.~Zhu,
\emph{The homotopy decomposition of the suspension of a non-simply-connected five-manifold},
Proc.\ Roy.\ Soc.\ Edinburgh Sect.\ A \textbf{156} (2026), no.~1, 93--121. \texttt{doi:10.1017/prm.2024.49}, \texttt{arXiv:2311.16642}.

\bibitem{hatcher}
A.~Hatcher,
\emph{Algebraic Topology},
Cambridge University Press, Cambridge, 2002.
Available at \texttt{https://pi.math.cornell.edu/\textasciitilde hatcher/AT/ATpage.html}.

\bibitem{kervaire:milnor}
M.~A.~Kervaire and J.~W.~Milnor,
\emph{Groups of homotopy spheres, I},
Ann.\ of Math.\ (2) \textbf{77} (1963), 504--537.

\bibitem{milnor:stasheff}
J.~W.~Milnor and J.~D.~Stasheff,
\emph{Characteristic Classes},
Ann.\ of Math.\ Studies, vol.~76,
Princeton University Press, Princeton, 1974.

\bibitem{spanier}
E.~H.~Spanier,
\emph{Algebraic Topology},
McGraw-Hill, New York, 1966.

\bibitem{wall}
C.~T.~C.~Wall,
\emph{Surgery on Compact Manifolds},
London Math.\ Soc.\ Monogr., vol.~1,
Academic Press, London, 1970;
2nd ed., Math.\ Surveys Monogr., vol.~69, Amer.\ Math.\ Soc., Providence, 1999.

\bibitem{alexander}
J.~W.~Alexander,
\emph{On the deformation of an $n$-cell},
Proc.\ Nat.\ Acad.\ Sci.\ \textbf{9} (1923), 406--407.

\bibitem{bks}
I.~Belegradek, S.~Kwasik, and R.~Schultz,
\emph{Codimension two souls and cancellation phenomena},
Adv.\ Math.\ \textbf{275} (2015), 1--46. \texttt{arXiv:0912.4874}.

\end{thebibliography}
\end{document}